\documentclass[reqno, a4paper]{amsart}

\usepackage[english]{babel}
\usepackage[T1]{fontenc}
\usepackage[applemac]{inputenc}

\usepackage{amsmath}
\usepackage{amsfonts}
\usepackage{amsthm}
\usepackage{mathrsfs}
\usepackage{lineno}
\usepackage{amssymb}
\usepackage{amscd}
\usepackage{amsrefs}

\usepackage[all]{xy}
\SelectTips{cm}{}

\usepackage{hyperref}

\newcommand{\hta}[1]{\hypertarget{#1}{\textrm (#1)}}
\newcommand{\hl}[1]{\hyperlink{#1}{\textrm (#1)}}
\newcommand{\hli}[1]{\hyperlink{#1}{\emph{\textrm (#1)}}}
\newcommand{\Ref}[1]{\ref{#1}}
\newcommand{\Refi}[1]{\emph{\ref{#1}}}

   {\end{list}}

\newtheorem{theorem}{Theorem}[section]
\newtheorem{lemma}[theorem]{Lemma}
\newtheorem{proposition}[theorem]{Proposition}
\newtheorem{corollary}[theorem]{Corollary}
\newtheorem{definition}[theorem]{Definition}

\newcommand{\Fs}{\ensuremath{\mathscr{F}}}
\newcommand{\Ys}{\ensuremath{\mathscr{Y}}}
\newcommand{\Ds}{\ensuremath{\mathscr{D}}}
\newcommand{\Ps}{\ensuremath{\mathscr{P}}}
\newcommand{\Cs}{\ensuremath{\mathscr{C}}}
\newcommand{\Xs}{\ensuremath{\mathscr{X}}}
\newcommand{\Gs}{\ensuremath{\mathscr{G}}}
\renewcommand{\Ps}{\ensuremath{\mathscr{P}}}

\newcommand{\Ts}{\ensuremath{\mathscr{T}}}

\newcommand{\Gbb}{\ensuremath{\mathbb{G}}}

\newcommand{\Tbb}{\ensuremath{\mathbb{T}}}
\newcommand{\Nbb}{\ensuremath{\mathbb{N}}}
\newcommand{\Zbb}{\ensuremath{\mathbb{Z}}}

\newcommand{\ab}{\ensuremath{\mathrm{ab}}}
\newcommand{\Ab}{\ensuremath{\mathrm{Ab}}}
\newcommand{\Gp}{\ensuremath{\mathrm{Gp}}}

\newcommand{\Ec}{\ensuremath{\mathcal{E}}}
\newcommand{\Fc}{\ensuremath{\mathcal{F}}}
\newcommand{\Pc}{\ensuremath{\mathcal{P}}}

\newcommand{\Gc}{\ensuremath{\mathcal{G}}}

\newcommand{\Mr}{\ensuremath{\it \mathcal{M}}}
\newcommand{\Qc}{\ensuremath{\mathcal{Q}}}
\newcommand{\Tc}{\ensuremath{\mathcal{T}}}

\newcommand{\Arr}{\ensuremath{\mathrm{Arr}}}

\newcommand{\Dom}{\ensuremath{\mathrm{Dom}}}
\newcommand{\Cod}{\ensuremath{\mathrm{Cod}}}
\newcommand{\br}{\ensuremath{{\bf \mathrm{r}}}}
\newcommand{\Eq}{\ensuremath{\mathrm{Eq}}}
\newcommand{\Ker}{\ensuremath{\mathrm{Ker}}}
\renewcommand{\ker}{\ensuremath{\mathrm{ker}}}
\newcommand{\Coker}{\ensuremath{\mathrm{Coker}}}
\newcommand{\coker}{\ensuremath{\mathrm{coker}}}

\newcommand{\Haus}{\ensuremath{\mathrm{Haus}}}

\newcommand{\Prof}{\ensuremath{\mathrm{Prof}}}
\newcommand{\Mono}{\ensuremath{\mathrm{Mono}}}
\newcommand{\RegEpi}{\ensuremath{\mathrm{RegEpi}}}

\newcommand{\rf}{\ensuremath{\mathrm{r}}}

\newcommand{\ConnComp}{\ensuremath{\mathrm{ConnComp}}}

\newcommand{\Set}{\ensuremath{\mathrm{Set}}}
\newcommand{\Top}{\ensuremath{\mathrm{Top}}}
\newcommand{\HComp}{\ensuremath{\mathrm{HComp}}}

\newcommand{\Ext}{\ensuremath{\mathrm{Ext}}}

\newcommand{\NExt}{\ensuremath{\mathrm{NExt}}}

\newcommand{\Triv}{\ensuremath{\mathrm{TExt}}}
\newcommand{\Split}[1]{\ensuremath{{\mathrm{Split}(#1)}}}
\newcommand{\Sp}{\ensuremath{\mathrm{Split}}}
\newcommand{\bt}{\ensuremath{\mathrm{t}}}
\newcommand{\Gal}{\ensuremath{\mathrm{Gal}}}

\newcommand{\Pres}{\ensuremath{\mathrm{Pres}}}

\newcommand{\ad}[4]{\xymatrix{
#1 \ar@/^2ex/[r]^-{#3}  \ar@{}[r]|-\bot &#2 \ar@/^2ex/[l]^-{#4}
}
}

\newcommand{\gal}[3]{\xymatrix{
#1 \ar@^{->}[rr]^-{#2}&& #3
}
}
\newcommand{\catequivalence}[4]{\xymatrix{
#1 \ar@/^2ex/[r]^-{#3}  \ar@{}[r]|-\cong &#2 \ar@/^2ex/[l]^-{#4}
}
}

\newcommand{\Pb}{\overline{\; \rule{0ex}{1,6ex} \cdot \; }}

\newbox\pullbackbox
\setbox\pullbackbox=\hbox{\xy 0;<1mm,0mm>: \POS(4,0)\ar@{-} (0,0) \ar@{-} (4,4)
\endxy}
\def\pullback{\copy\pullbackbox}
\newbox\pushoutbox
\setbox\pushoutbox=\hbox{\xy 0;<1mm,0mm>: \POS(0,4)\ar@{-} (0,0) \ar@{-} (4,4)
\endxy}

\newdir{|>-}{!/  4.5pt / @{ }*@{ }*@{ }*!/ 0pt / @{|}*!/ -4.5pt / :(1,-.2)@^{>}*!/ -4.5pt / :(1,+.2)@_{>}}

\newdir{-|>}{!/ 9pt / @{|}*!/ 4.5pt /  :(1,-.2)@^{>}*!/ 4.5pt /  :(1,+.2)@_{>}*@{ }*@{ }*@{ } }

\newdir{ >}{!/  8.5pt / @{ }*@{ }*@{ }*@{>}}

\begin{document}

\author{M.~Duckerts-Antoine}
\address{Centre for Mathematics, University of Coimbra, Departamento de Matem\'atica, Apartado 3008, 3001-501 Coimbra, Portugal}
\email{mathieud@mat.uc.pt}

\thanks{This work was partially supported by the Universit\'e catholique de Louvain, by the FCT - Funda\c{c}\~ao para a Ci\^encia e a Tecnologia - under the grant number SFRH/BPD/98155/2013, and by the Centre for Mathematics of the University of Coimbra -- UID/MAT/00324/2013, funded by the Portuguese
Government through FCT/MEC and co-funded by the European Regional Development Fund through the Partnership Agreement PT2020. This article is partly based on our thesis \cite{Mathieu} whose defense was held at the Universit\'e catholique de Louvain.}

\title[Fundamental group functors in descent-exact homological cat.]{Fundamental group functors in\\descent-exact homological categories}

\begin{abstract} We study the notion of fundamental group in the framework of descent-exact homological categories. This setting is sufficiently wide to include several categories of ``algebraic'' nature such as the almost abelian categories, the semi-abelian categories, and the categories of topological semi-abelian algebras. For many adjunctions in this context, the fundamental groups  are described by generalised Brown-Ellis-Hopf formulae for the integral homology of groups. 
\end{abstract}

\maketitle

\section{Introduction}

In \cite{DuckertsEveraertGran}, we pursued the study of the categorical notion of fundamental group introduced in \cite{Janelidze:Hopf} and provided a generalised version of the Hopf formula for the description of the fundamental group within the semi-abelian context \cite{Janelidze-Marki-Tholen}. Examples of semi-abelian categories are the categories of groups, Lie algebras, compact Hausdorff groups, crossed modules, and similar non-abelian structures. In the present work, we define and study higher fundamental groups within the wider context of descent-exact homological categories \cite{BB}. This allows us to cover a lot of other categories, let us just mention here the categories of topological groups, locally compact abelian groups, and Banach spaces (and bounded linear maps).

In order to understand what is a descent-exact homological category, let us first recall the well-known Tierney's description of abelian categories:
\[
\text{abelian =  additive + Barr-exact} 
\]
Let us also recall that a category is Barr-exact \cite{Barr} when it is regular (i.e. is finitely complete and has coequalizers of every kernel pair) and every internal equivalence is the kernel pair of some morphism. It turns out that in a Barr-exact category $\Cs$, every regular epimorphism $f \colon E \to B$ is effective for descent, which means that the pullback functor $p^* \colon (\Cs \downarrow B)\to (\Cs \downarrow E)$ is monadic. This can be viewed as a form of exactness condition on a category, that we call here descent-exactness  (see \cite{Garner-Lack} for a general notion of exactness). Thus, the kind of categories we consider are
\[
\text{$\underbrace{\text{pointed +  protoadditive + regular}}_{\text{homological}}$ + descent-exact}
\] 
where, in presence of the other axioms, the protomodularity \cite{Bourn1991} condition can be equivalently expressed by saying that the split short five lemma holds. These axioms have numerous consequences. For instance homological lemmas such as the snake lemma or the 3 by 3 lemma still hold in this general context (see the monograph \cite{BB} for a general introduction to homological and semi-abelian categories). Note that the descent-exact homological category of topological groups is neither additive nor exact.

The notion of fundamental group is related to the concept of normal extension coming from the categorical Galois theory \cite{Janelidze:Pure}. In order to give an idea of this relation, let us consider what happens in a simple case. With respect to the reflection of the category $\Gp$ of groups into the category $\Ab$ of abelian groups given by the abelianisation functor $\ab \colon \Gp \to \Ab$ (we write $\eta \colon 1 \Rightarrow \ab$ for the unit of this reflection), one says that a surjective homomorphism $p \colon E \to B$ (an extension)  is a normal extension if the first projection $\pi_1$ of its kernel pair
\[
\xymatrix{
\Eq(p)\times_E \Eq(p) \ar@<1ex>[rr]^-{p_1}  \ar@<-1ex>[rr]_-{p_2} \ar@{.>}[rr]|-{\tau} && \Eq(p) \ar@(ul,ur)[]^-{\sigma} \ar@<1ex>[rr]^-{\pi_1} \ar@<-1ex>[rr]_-{\pi_2}& &E  \ar@{{ >}.>}[ll]|-{\delta}
}
\]
is such that the naturality square 
\[
\xymatrix{
E  \ar[d]_-p \ar[r]^-{\eta_E} &\ab(E) \ar[d]^-{\ab(p)}\\
B \ar[r]_-{\eta_B} & \ab(B)
}
\]
is a pullback. As shown in \cite{Janelidze:Pure}, this property is also equivalent to the fact that $p$ is a central extension of groups: the kernel $\Ker(p)$ of $p$ is a subgroup of the center of $E$.  Now it is always possible to turn an extension $p\colon E\to B$ into a normal extension: taking the quotient of $E$ by $[\Ker(p),E]$, the induced factorisation $I_1(p)$ of $p$
\[
\xymatrix{
E \ar[r]^-p \ar[d]& B\\
\frac{E}{[\Ker(p),E]} \ar@{.>}[ru]_-{ I_1(p)}&
}
\]
is a normal extension. This procedure gives a reflection $I_1 \colon \Ext(\Gp) \to \NExt(\Gp)$ of the category $\Ext(\Gp)$ of extensions of groups  into the category $\NExt(\Gp)$ of normal extensions of groups (we shall write $\eta^1$ for the unit of this reflection). Now let us assume that $p\colon E\to B$ is the normalisation of some free presentation $f\colon F \to B$. In that case, it turns out that the Galois group $\Gal(p,0)$ of $p$, defined as the group of automorphisms of $0$ in the groupoid
\[
\xymatrix{
\ab(\Eq(p)\times_E \Eq(p))\ar[rr]^-{\ab(\tau)} && \ab(\Eq(p)) \ar@(ul,ur)[]^-{\ab(\sigma)} \ar@<1ex>[rr]^-{\ab(\pi_1)}  \ar@<-1ex>[rr]_-{\ab(\pi_2)}  & & \ab(E)  \ar@{.>}[ll]|-{\ab(\delta)},
}
\]  
is an invariant of $B$: the first fundamental group $\pi_1(B)$ of $B$. This extends to a functor $\pi_1 \colon \Gp \to \Ab$. It was shown in \cite{EGoeVdL2} that this functor can be viewed as a left (pointwise) Kan extension in two different ways: as a Kan extension of the composite functor $\Gal(I_1(-),0)$ along the codomain functor $\Cod$:
\[
\xymatrix{
&\Ext(\Gp) \ar[dl]_-{\Cod} \ar[dr]^-{\Gal(I_1(-),0)}& \\
\Gp \ar@{.>}[rr]_-{\pi_1} \ar@{}[urr]|(0.4){\Rightarrow}&& \Ab
}
\]
or as a Kan extension of the composite functor $\Ker \circ I_1$ along $\Cod$ where $\Ker$ is the kernel functor:
\[
\xymatrix{
\Ext(\Gp) \ar[r]^-{I_1} \ar[d]_-{\Cod}& \NExt(\Gp) \ar[d]^-{\Ker}\\
\Gp \ar[r]_-{\pi_1} \ar@{}[ru]|-{\Rightarrow}& \Ab,
}
\]
that is as a satellite of the normalisation functor $I_1$.
Moreover, it was known since \cite{Janelidze:Hopf} that the first fundamental group of $B$ can be described, directly from the presentation $f\colon F \to B$, by the famous Hopf formula \cite{Hopf} for the second integral homology group of $B$:
\[
\pi_1(B)\cong \Gal(I_1(f),0) \cong  \frac{[F,F] \cap \Ker(f)}{[\Ker(f) ,F]} \cong H_2(B , \Zbb)
\] 
Let us go a bit further and let us remark that the data
 \[
\Gamma = (\Cs,\Xs,I,\eta, \Ec),
\]
with $\Cs=\Gp$, $\Xs=\Ab$, $I=\ab$, and $\Ec$ the class $\RegEpi(\Gp)$ of regular epimorphisms in $\Gp$, is actually what is called a closed Galois structure (see section 2.3 for a precise definition). Similarly, the data
\[
\Gamma_1 = (\Ext_\Ec(\Cs),\NExt_\Gamma(\Cs),I_1,\eta^1,\Ec^1)
\]
given by $\Ext_\Ec(\Cs)=\Ext(\Gp)$, $\NExt_\Gamma(\Cs)=\NExt(\Gp)$,  and $\Ec^1$ the class of double extensions \cite{Janelidze:Double} (a class of arrows in $\Ext_\Ec(\Cs)$ defined relatively to $\Ec$), is also a closed Galois structure that satisfies some of the conditions that $\Gamma$ satisfies. Inductively this leads to a tower of closed Galois structures
\[
\Gamma_n = (\Ext_\Ec^n(\Cs),\NExt_\Gamma^n(\Cs),I_n,\eta^n,\Ec^n)
\]
satisfying a suitable set of axioms. It is in that case possible to define higher fundamental groups functors in the following way: for $n\geq 1$, the $n^{\text{th}}$-fundamental group functor $\pi_n^{\Gamma}$, with respect to the Galois structure $\Gamma$ can be defined as the pointwise right Kan extension of $\Gs_{n}^\Gamma=\Dom^{n-1}\left( \Gal_{\Gamma_{n-1}}(I_n(-),0)\right)$ along the functor $\Cod^n \colon  \Ext^n_\Ec (\Cs) \to \Cs$: 
\[
\xymatrix{
&\Ext^n_\Ec(\Cs) \ar[dl]_-{\Cod^n} \ar[dr]^-{\Gs_{n}^\Gamma}& \\
\Cs \ar@{.>}[rr]_-{\pi_n^\Gamma} \ar@{}[urr]|(0.4){\Rightarrow}&& \Ab(\Xs)
}
\]
where $\Ab(\Xs)$ $(=\Ab)$ is the category of internal abelian groups in $\Xs$, and $\Dom^n$ and $\Cod^n$ are iterated versions of the  domain and codomain functors. The other definition in term of satellites and the higher Hopf formulae \cite{Brown-Ellis} for the description of the higher fundamental groups are also available. We have the isomorphisms 
\[
\pi_n(B) \cong \Gs_n^\Gamma(F) \cong H_{n+1}(B,\Zbb)
\]
where $F$ is any $n$-fold projective presentation of $B$ (see the end of the section 2.6 for a definition). 

Of course (with the exception of some minor details), nothing above is special to the category of groups and the fundamental group functors can be defined and studied in any descent-exact homological category with enough projective objects provided that the basic reflector $I$ is sufficiently good. By this, we mean that $I$ preserves pullbacks of type
\[
\xymatrix{
A \ar[r]  \ar[d]& B\ar[d]^-g\\
D\ar[r]_-h&C
}
\]
where $g$ is a split epimorphism and $h$ a regular epimorphism. 

It should be understood here that what we seek is a good definition of homology in our general setting. Some previous works in the field tend to show that the notion of fundamental group effectively provides the right notion. In particular we want to mention \cite{EGVdL} in which generalized Hopf formulae were given for the Barr-Beck cotriple homology when the coefficient functor is the reflection of a monadic semi-abelian category  into one of its Birkhoff subcategories (a subcategory closed under quotients and subobjects). This led the author of \cite{EverHopf} to take the Hopf formulae as the definition of homology objects in any semi-abelian category with enough regular projectives (with respect to a Birkhoff reflection once again). Still in the same situation, homology functors were already proved to be satellites in \cite{GVdL2}. Note that the latter approach to homology has the advantage that it doesn't require  at all projective objects.

The last important aspect of the present work concerns a simplification of the formulae for fundamental groups functors that occurs when the reflector $I$ factors as a reflector of the same kind followed by a protoadditive reflector. The notion of protoadditive functor was proposed in \cite{EG-honfg} as the suitable replacement of the notion of additive functor in the context of homological categories. In particular several connections with non-abelian homology were already studied  in the series of papers \cite{EG-honfg,Everaert-Gran-TT,DuckertsEveraertGran}. A protoadditive functor between homological categories is a functor which preserves split short exact sequences, i.e. short exact sequences
\[
\xymatrix{
0 \ar[r]& \Ker(f) \ar[r]^-{\ker(f)}  &A \ar@<0.5ex>[r]^-f & B \ar@<0.5ex>@{.>}[l]^-s   \ar[r]& 0
}
\]
where $f$ admits a section $s$. Such functors allow the introduction  in the Hopf formulae of some homological closure operator \cite{BG:Torsion} associated with the reflection. Particularly nice formulae are obtained when the reflector is additive or comes from a torsion theory (see the last section). The refined Hopf formula in the special case of the first fundamental group functor was the essential content of \cite{DuckertsEveraertGran}.

\tableofcontents

\section{Galois structures and extensions}

\subsection{Descent-exact homological categories}

A \emph{descent-exact homological category} is a regular pointed protomodular category in which every regular epimorphism is an effective descent morphism. We are going to recall what these terms mean and provide three main classes of such categories.

\paragraph{Pointed categories} A category is \emph{pointed} when it has an object which is both initial and terminal. This \emph{zero object} is denoted by $0$. For two objects $X$ and $Y$, we also write $0$ for the unique morphism from $X$ to $Y$ which factors through the zero object. 

\paragraph{Regular categories} A category $\Cs$ is \emph{regular}  \cite{Barr-Grillet-vanOsdol}  if it is finitely complete, every morphism $f$ can be factorised as a regular epimorphism followed by a monomorphism and regular epimorphisms are stable under pullbacks. 

\paragraph{Protomodular categories} A pointed category $\Cs$ is \emph{protomodular} \cite{Bourn1991} if 
\begin{itemize}
\item it has pullbacks along split epimorphisms;
\item it satisfies the \emph{split short five lemma}, i.e.\  given any commutative diagram
\[
\xymatrix{
\Ker(f) \ar[r]^-{\ker(f)}  \ar[d]_-u & A \ar@<0.5ex>[r]^-f \ar[d]^-v & B\ar[d]^-w \ar@<0.5ex>[l]^-s \\
\Ker(f')\ar[r]_-{\ker(f')}  & A' \ar@<0.5ex>[r]^-{f'} & B'\ar@<0.5ex>[l]^-{s'} 
}
\]
with $f\circ s=1_B$, $f'\circ s' = 1_{B'}$, if $u$, $v$ are isomorphisms then $v$ is also an isomorphism.
\end{itemize}
Note that the notion of protomodularity is generally defined more conceptually by a property of the so-called fibration of points (with no need of a zero object). A pointed protomodular category is \emph{unital} \cite{B0}, i.e.\ for any $X$ and $Y$, the pair of morphisms
\[
\xymatrix{
X  \ar[rr]^-{\iota_X =(1_X,0)} && X\times Y && Y \ar[ll]_-{\iota_Y=(0,1_Y)}. 
}
\]
is \emph{jointly extremal-epimorphic} (if $\iota_X$ and $\iota_Y$ factor through the same monomorphism, this monomorphism is  an isomorphism). Let us mention here that the product of two objects $X$ and $Y$ may be computed in a pointed category with pullbacks along split epimorphisms as the pullback of $X \to 0$ along $Y\to 0$. If the category has finite limits (equalisers actually), the pair of morphisms $(\iota_X,\iota_Y)$ is jointly epimorphic ($f \circ \iota_X = f' \circ \iota_X$ and $f \circ \iota_Y = f' \circ \iota_Y$ implies $f=f'$) and an object can only have one internal group structure: to be an internal group is a property. There is a useful criterion which can be used to recognise if a category is protomodular: if $\Xs \to \Cs$ is a conservative functor which preserves pullbacks and $\Cs$ is protomodular, then $\Xs$ is also protomodular. In a pointed protomodular category, we write
\[
\xymatrix{
0 \ar[r]   & K \ar[r]^-k  & A  \ar[r]^-f & B  \ar[r] & 0 
}
\]
when $f$ is a regular epimorphism and $k$ its kernel. Since in this context a regular epimorphism is always the cokernel of its kernel, we may call a sequence as above a \emph{short exact sequence}.

\paragraph{Morphisms of effective descent} A morphism $f\colon A\to B$ in a category $\Cs$ with pullbacks is an \emph{effective descent morphism} if the pullback functor 
\[
f^* \colon (\Cs \downarrow B) \to (\Cs \downarrow A)
\]
 is monadic. In a regular category, every effective descent morphism is a regular epimorphism. Note that the converse holds, in particular, in any \emph{Barr-exact} category, i.e.\ a regular category in which every internal equivalence relation is \emph{effective} (the kernel pair of some morphism). Thus the descent-exactness condition, that says that every regular epimorphism is effective for descent, can be viewed as a weakened version of the the Barr-exactness condition. The reader is invited to consult  \cite{Janelidze-Sobral-Tholen} for a nice introduction to descent theory.

\paragraph{Almost abelian categories} An  \emph{almost abelian} category \cite{Rump0} can be described in our terms as a category which is homological and \emph{cohomological} (its dual category is homological). Every almost abelian category is descent-exact \cite[Section 4.4]{Gran-Ngaha}. The original definition is the following: a category is almost abelian if it is additive, has kernels and cokernels and moreover  normal epimorphisms are pullback-stable and normal monomorphisms are pushout-stable. The categories of locally compact groups, normed vector spaces, Banach spaces (and bounded linear maps), Fr\'echet spaces are almost abelian. Of course, every abelian category is almost abelian.
 
\paragraph{Semi-abelian categories} A category is \emph{semi-abelian} \cite{Janelidze-Marki-Tholen} when it is Barr-exact, pointed protomodular and has binary coproducts. Among the basic examples of semi-abelian categories are the categories of groups, Lie algebras, rings, crossed modules, compact Hausdorff groups. Semi-abelian varieties (in the sense of universal algebra) have been completely characterised:

 \begin{theorem}\label{characterization-varieties-semi-abelian} \cite{Bourn-Janelidze} A finitary algebraic theory $\Tbb$ is semi-abelian, i.e.\ has a semi-abelian category $\Tbb(\Set)$ of models, precisely when, for some natural number $n$, the theory $\Tbb$ contains
  \begin{itemize}
  \item a unique constant $0$;
  \item $n$ binary operations $\alpha_1(X,Y)$, $\hdots$,  $\alpha_n(X,Y)$ satisfying $\alpha_i(X,X)=0$;
  \item an $(n+1)$-ary operation $\theta(X_1,\hdots,X_{n+1})$ satisfying
  \[
  \theta(\alpha_1(X,Y),\hdots,\alpha_n(X,Y),Y)=X.
  \]
  \end{itemize}
  \end{theorem}
  Using this theorem, it is easy to show that the category of groups, for instance, is semi-abelian. Indeed it is easy to check that the above equations are satisfied with $1$ as the unique constant in the theory, $n=1$,  $\alpha_1(X,Y)=X.Y^{-1}$ and $\theta(X,Y)=X.Y$. More recently, it was shown that the category of cocommutative Hopf algebras over a field of characteristic zero is also semi-abelian \cite{GKV}. It should be noted that a category $\Cs$ is in fact abelian if and only if both $\Cs$ and $\Cs^{op}$ are semi-abelian.

\paragraph{Topological semi-abelian varieties} A \emph{topological semi-abelian variety} \cite{Borceux-Clementino} is the category of models in $\Top$ of some semi-abelian theory $\Tbb$. An example is the category of topological groups $\Gp(\mathsf{Top})$ which is known to be not exact and not additive. Therefore $\Gp(\Top)$ is neither a semi-abelian category nor an almost abelian category. Nevertheless, it is known that every topological semi-abelian variety is (cocomplete and) homological \cite[Theorem 50]{Borceux-Clementino} and descent-exact \cite[Section 4.5]{Gran-Ngaha}. Note that the categories of models of semi-abelian theories in the category of compact Hausdorff spaces are themselves semi-abelian categories \cite[Theorem 50]{Borceux-Clementino}.

\paragraph{Torsion theories} A \emph{torsion theory} \cite{Dickson,Janelidze-Tholen-TT} in a pointed category $\Cs$ is a couple $(\Ts,\Fs)$ of full replete subcategories of $\Cs$ such that:
 
 \begin{itemize}
 \item every morphism $f\colon T \to F$ with $T \in \Ts$ and $F\in \Fs$ is zero;
 \item for any object $A$ in $\Cs$ there exists a short exact sequence:
 \[
 \xymatrix{
 0 \ar[r]  & T\ar[r]^k & A \ar[r]^f& F\ar[r] &0 
 }
 \]
 with $T \in \Ts$, $F\in \Fs$.
 \end{itemize}
 For a torsion theory $(\Ts,\Fs)$ in $\Cs$, the subcategories $\Ts$ and $\Fs$ are called, respectively, the \emph{torsion part} and the \emph{torsion-free part} of the torsion theory. As one can easily check from the definition, $\Ts$ is in fact a full replete normal mono-coreflective subcategory of $\Cs$ and $\Fs$ a full replete normal epi-reflective subcategory of $\Cs$. For $\Mr$ a class of monomorphisms and $(\Ts,\Fs)$ a torsion theory in a pointed category $\Cs$, one says that the torsion theory is \emph{$\Mr$-hereditary} if $\Ts$ is closed under $\Mr$-subobjects. In \cite{Gran-Lack}, torsion-free subcategories of semi-abelian categories were characterized as being precisely the (descent-exact) homological categories with binary coproducts and stable coequalisers. It turns out that all the three classes of examples of descent-exact homological categories we have given are torsion free subcategories (of their exact completions). But not every descent-exact homological category is a torsion-free subcategory of some semi-abelian category.

\subsection{Categories with a class of extensions}

As (implicitly) stated in the introduction, we shall work with pairs $(\Cs,\Ec)$ that satisfy some of the conditions that are satisfied by the pair $(\Cs,\RegEpi(\Cs))$ for $\Cs$ a descent-exact homological category and $\RegEpi(\Cs)$ the class of regular epimorphisms of $\Cs$. 

\paragraph{Axioms on extensions} Let $\Cs$ be a pointed protomodular category and $\Ec$ a subclass of $\RegEpi(\Cs)$. We shall denote the class of morphisms in $\Ec$ which are also split epimorphisms by $\Split{\Ec}$ and the full subcategory of the category $\Arr(\Cs)$ of arrows  whose objects are morphisms in $\Ec$ by $\Ext_\Ec(\Cs)$.  We shall write $(\Cs \downarrow_{\Ec} B)$ for the full subcategory of the comma category $(\Cs \downarrow B)$ whose objects are the arrows in $\Ec$ with codomain $B$. The pair $(\Cs,\Ec)$ satisfies
\begin{description}
\item[\hta{E1}] if $\Ec$ contains the isomorphisms in $\Cs$;
\item[\hta{E2}] if pullbacks of morphisms in $\Ec$ exists and are in $\Ec$;
\item[\hta{E3}] if $\Ec$ is closed under composition;
\item[\hta{E4}] if $g \circ f$ in $\Ec$ implies that $g$ is also in $\Ec$;
\item[\hta{E5}] if given a commutative diagram in $\Cs$
\[
\xymatrix{
 \Ker(a)  \ar[r]^-{\ker(a)}   \ar[d]_-k  & A_1 \ar[r]^-a   \ar[d]_-{f}  &A_0  \ar@{=}[d] \\
  \Ker(b) \ar[r]_-{\ker(b)}    & B  \ar[r]_-b  & A_0 
}
\]
with $a$ and  $k$ in $\Ec$, then necessarily $f$ is also in $\Ec$. 
\item[\hta{M}] if every morphism in $\Ec$ is \emph{monadic}, i.e.\ for all $f\colon A\to B$ in $\Ec$, the change-of-base functor $f^*\colon (\Cs \downarrow_\Ec B) \rightarrow (\Cs\downarrow_\Ec A)$ is monadic. 
\end{description}
Let us recall here that for every $f\colon A \to B$ in $\Ec$ the functor $f^*$ has a left adjoint $f_!$ (composition with $f$). This determines a category of Eilenberg-Moore algebras $(\Cs\downarrow_{\Ec} A)^{T^f}$ for the corresponding monad $T^f=f^*f_!$ and $f$ is monadic precisely when the comparison functor
\[
K^{T^f}\colon (\Cs\downarrow_{\Ec} B)\to (\Cs\downarrow_{\Ec} A)^{T^f}
\]
is an equivalence of categories. 

\paragraph{Double extensions} Given a pair  $(\Cs,\Ec)$ which satisfies \hl{E2} as above, it is possible to define a good class of morphisms $\Ec^1$ in $\Ext_\Ec(\Cs)$ as follows: a morphism $(f_1,f_0)\colon a \to b$ in $\Ext_\Ec(\Cs)$ is in $\Ec^1$ (and is called a \emph{double $\Ec$-extension} \cite{Janelidze:Double}) if all morphisms in the commutative diagram
\[
\xymatrix{
A_1 \ar[rd]|-{\langle a , f_1\rangle} \ar@/^/[rrd]^-{f_1} \ar@/_/[rdd]_-{a} \\
& A_0\times_{B_0} B_1 \ar[r]^{p_2} \ar[d]_-{p_1}&B_1 \ar[d]^-b \\
&A_0 \ar[r]_-{f_0}& B_0
}
\]
are in $\Ec$. 

\begin{theorem} Going from $(\Cs,\Ec)$ to $(\Ext_\Ec(\Cs),\Ec^1)$ preserves the following sets of axioms: 
\begin{enumerate}
\item $\{\hli{E1}-\hli{E3}\}$;
\item  $\{\text{pointed, protomodular, $\hli{E1}-\hli{E5}$, $\Ec\subseteq \RegEpi(\Cs)$}\}$;
\item $\{\text{pointed, protomodular, $\hli{E1}-\hli{E5}$, \hli{M}}\}$.
\end{enumerate}
\end{theorem}
\begin{proof}
For the assertion 1 see \cite[Proposition 3.5]{EGVdL} or \cite[Proposition 1.6]{EGoeVdL} and for the assertion 2 see \cite[Proposition 1.8]{EverHopf}. Let us now consider the last assertion. Let us first remark that the axioms imply in particular that every morphism in $\Ec$ is a regular epimorphism (or equivalently here) a  normal epimorphism (see for instance \cite[Remark 1]{EverMaltsev}). Then one concludes by  \cite[Lemma 9]{EverMaltsev} and assertion~2. Indeed our axiom \hl{E5} is easily seen to imply the axiom (E5) used in this last reference (see the discussion in \cite[page 153]{EGoeVdL}).
\end{proof}

\paragraph{Isomorphisms of pairs} Given two pairs $(\Cs,\Ec)$ and $(\Xs,\Fc)$, we shall say that an isomorphism $F\colon \Cs \to \Ds$ is an \emph{isomorphism of pairs}
\[
F\colon (\Cs,\Ec) \to (\Xs,\Fc)
\]
if $F(\Ec)=\Fc$.

\subsection{Galois structures}

A structure $\Gamma=(\Cs,\Xs, I,H ,\eta,\epsilon, \Ec,\Fc)$ is a \emph{closed Galois structure} \cite{Janelidze:Recent} if
\begin{itemize}
\item 
\[
\ad{\Cs}{\Xs}{I}{H}
\]
is an adjunction with unit and counit 
\[
\eta\colon 1_\Cs \Rightarrow H\circ I, \qquad \epsilon \colon I\circ H \Rightarrow 1_\Xs,
\]
\item the pairs $(\Cs,\Ec),(\Xs,\Fc)$ satisfy \hl{E1}, \hl{E2} and \hl{E3};
\item $I(\Ec)\subseteq \Fc$ and $H(\Fc) \subseteq \Ec$;
\item  $\eta_C\colon C \to HI(C)$ is in $\Ec$ for all $C$ in $\Cs$;
\item $\epsilon_X \colon IH(X) \to X$ is an isomorphism for all $X$ in $\Xs$.
\end{itemize}

We shall denote a closed Galois structure $\Gamma$ by
\[
 \gal{(\Cs,\Ec)}{\Gamma =(I,H,\eta,\epsilon)}{(\Xs,\Fc)}
\]
and we shall assume that $H$ is an inclusion and $\epsilon_X$ an identity for all $X$.
Of course it is possible to compose closed Galois structures like we can compose adjunctions.

\paragraph{Types of extensions} With respect to a closed Galois structure $\Gamma$ as above, an extension $f\colon A\to B$ is said to be 
\begin{enumerate}
\item
 \emph{trivial} if the naturality square
\[
\xymatrix{
A \ar[r]^-{\eta_A}\ar[d]_-f&HI(A)\ar[d]^-{HI(f)}\\
B\ar[r]_-{\eta_B} &HI(B)
}
\]
is a pullback;
\item 
\emph{normal} if it is a monadic extension and if $f^*(f)$ is trivial.
\end{enumerate}
When the context is not clear, we speak of $\Gamma$-trivial and of $\Gamma$-normal extensions. We shall also write $\Triv_\Gamma(\Cs)$ and $\NExt_\Gamma(\Cs)$ for the categories of $\Gamma$-trivial extensions and $\Gamma$-normal extensions respectively (considered as full subcategories of $\Ext_\Ec(\Cs)$).

In the following we shall be particularly interested  in reflectors $I$ which preserve pullbacks of the type \Ref{special-pullback}, i.e.\ pullbacks
\begin{equation}
\label{special-pullback}
\vcenter{
\xymatrix{
A \ar[r]  \ar[d]& B\ar[d]^-g\\
D\ar[r]_-h&C
}}
\end{equation}
where $h$ is in $\Ec$ and $g$ in $\Split{\Ec}$. The main examples of such functors are the so-called Birkhoff reflectors and the protoadditive reflectors.

\paragraph{Birkhoff reflector} The reflector $I$ is a \emph{strongly $\Ec$-Birkhoff reflector}  \cite{EverHopf,EGVdL} if, for every morphism $f\colon A \to B$ in $\Ec$, all arrows in the diagram
\[
\xymatrix{
A \ar@/^/[rdr]^-{\eta_A} \ar@/_/[ddr]_-f \ar@{.>}[rd]|-{\langle f, \eta_A\rangle} \\
&B \times_{HI(B)} HI(A) \ar[d] \ar[r]&HI(A)\ar[d]^-{HI(f)}\\
&B\ar[r]_-{\eta_B}&HI(B)
}
\]
are in $\Ec$. If $\Cs$ is a semi-abelian variety and $\Xs$ a subvariety, the reflector $I$ is necessarily a strongly $\RegEpi(\Cs)$-Birkhoff reflector. It suffices to remember that a subvariety is precisely a class of algebras which is closed under quotients, products and subobjects. This is the type of functors which was considered in the works \cite{EGVdL,EverHopf,EverMaltsev}.  The proof that these functors preserve pullbacks of type \ref{special-pullback} can be found in \cite[Lemma 4.4]{EGVdL}.

\paragraph{Protoadditive functors} A functor $F \colon \Cs\to \Xs$ between pointed protomodular categories is \emph{protoadditive} \cite{EG-honfg}  if it preserves \emph{split short exact sequences}:
if one has a split short exact sequence in $\mathcal{C}$
\[
\xymatrix{
0 \ar[r]& \Ker(f) \ar[rr]^-{\ker(f)}  &&A \ar@<0.5ex>[rr]^-f && B \ar@<0.5ex>@{.>}[ll]^-s   \ar[r]& 0
}
\]
($s$ a section of $f$), then its image by $F$ is also a split short exact sequence in $\Xs$:
\[
\xymatrix{
0 \ar[r]& F(\Ker(f)) \ar[rr]^-{F(\ker(f))}  &&F(A) \ar@<0.5ex>[rr]^-{F(f)} && F(B) \ar@<0.5ex>@{.>}[ll]^-{F (s)} \ar[r]& 0
}
\]
This notion was introduced to extend the notion of additive functor to a non-abelian context. They also have been used in relation to homology in \cite{Everaert-Gran-TT,DuckertsEveraertGran}.
Let us recall from \cite{BJK2} that a \emph{protosplit monomorphism} in a pointed protomodular category $\Cs$ is a normal monomorphism $k\colon K\to A$ that is the kernel of a split epimorphism. Some examples of protoadditive functors are provided by the following
\begin{theorem}\cite{Everaert-Gran-TT} For $(\Ts,\Fs)$ a torsion theory in a homological category $\Cs$, the following conditions are equivalent:
\begin{enumerate}
\item the torsion subcategory $\Ts$ is $\Mr$-hereditary, for $\Mr$ the class of protosplit monomorphisms;
\item the reflector $\Cs \to \Fs$ is protoadditive.
\end{enumerate} 
\end{theorem}

\paragraph{Trivialisation functor} For a closed Galois structure $\Gamma$ such that the reflector $I$ preserves pullbacks of the kind \Ref{special-pullback}, one always gets an induced Galois structure
\[
\Gamma_\Sp\colon \gal{(\Cs,\Split \Ec)}{(I,H,\eta,\epsilon)}{(\Xs,\Split \Fc)}
\]
for which the inclusion functor 
\[
\xymatrix{
\Triv_{\Gamma_\Sp}(\Cs) \ar[r]^-{\tilde H_1} &  \Ext_{\Split{\Ec}}(\Cs)
}
\]
admits a left adjoint $T_1$. We shall write $\tilde \eta^1$ for the unit of the adjunction $T_1 \dashv \tilde H_1$. The trivialization $T_1(f)$ of an extension $f\colon A\to B$ in $\Split{\Ec}$ is given by $\eta_B^*(HI(f))$ as in the diagram
\[
\xymatrix{
A \ar@/^/[rdr]^-{\eta_A} \ar@/_/[ddr]_-f \ar@{.>}[rd]|-{\langle f, \eta_A\rangle} \\
&B \times_{HI(B)} HI(A) \ar[d]^-{T_1(f)} \ar[r]&HI(A)\ar[d]^-{HI(f)}\\
&B\ar[r]_-{\eta_B}&HI(B).
}
\]
Indeed $\tilde \eta^1_f =(\langle f, \eta_A\rangle , 1_B) \colon f \to T_1(f)$ and we shall write $\overline{\tilde \eta}^1_f =\langle f, \eta_A\rangle$. 

\paragraph{Normalisation functor} When every extension is monadic, the inclusion functor 
\[
\xymatrix{
\NExt_\Gamma(\Cs)  \ar[r]^-{ H_1} &  \Ext_\Ec(\Cs)
}
\]
also has a left adjoint $I_1$ and we write $\eta^1$ for the unit of the adjunction $I_1 \dashv H_1$. The normalization $I_1(f)$ of an extension $f\colon A \to B$ is given by the commutative diagram
\begin{equation}\label{magicdiagram}
\vcenter{
\xymatrix{
\Eq(f)  \ar[rr]^-{\pi_2}  \ar[dr]^(0.6){\overline{\tilde \eta}^1_{\pi_1}}   \ar[ddd]_(0.6){\pi_1}   && A \ar[rd]^(0.6){\overline{\eta}^1_f}   \ar@{.>}[ddd]_(0.6)f  &\\
& T_1[f] \ar[rr]   \ar[ldd]^-{T_1(\pi_1)}  &&I_1[f] \ar[ldd]^-{I_1(f)}  \\\\
A  \ar[rr]_-f  &&B
}
}
\end{equation}
in which the three squares are pullbacks and pushouts. In fact, the morphism $\overline{\tilde \eta}^1_{\pi_1} \colon \pi_1 \to T_1(\pi_1)$ is a morphism in  $(\Cs\downarrow_{\Ec} A)^{T^f}$ and $\overline{\eta}^1_f\colon f \to I_1(f)$ the corresponding morphism in $\Cs \downarrow_\Ec B$. Using the above diagram, it easy to show that $I(I_1(f))=I(f)$. Note also that $I_1$ restricted to $\Ext_{\Split{\Ec}}(\Cs)$ is $T_1$.  See \cite{ED-NormalExt} for a detailed explanation.

\paragraph{Radicals} The reflector $I \colon \Cs \to \Xs$ induces a radical, i.e. a normal subfunctor $\mu \colon [-]_{\Gamma,0} \to 1_\Cs$ such that $[X/[X]_{\Gamma,0}]_{\Gamma,0}=0$ for all $X$ in $\Cs$,    given by 
\[
\xymatrix{
0 \ar[r] & [-]_{\Gamma,0} \ar[r]^-{\mu} & 1_{\Cs}  \ar[r]^-{\eta} & HI \ar[r] & 0
}
\]
There are relations between some properties of this radical and properties of the reflector:
\begin{lemma}\label{radref} Let 
\[
\Gamma \colon \gal{(\Cs,\Ec)}{(I,H,\eta,\epsilon)}{(\Xs,\Fc)}
\]
be such that
 \begin{itemize}
 \item $\Cs$ is pointed protomodular;
 \item $(\Cs,\Ec)$ satisfies \hli{E4}, \hli{$E5$} and $\Ec\subseteq \RegEpi(\Cs)$.
 \end{itemize}
 Then
\begin{itemize}
\item $I$ preserves pullbacks of the form \Refi{special-pullback} if and only if $[-]_{\Gamma,0}$ preserves them;
\item $I$ is protoadditive if and only if $[-]_{\Gamma,0}$ is protoadditive.\qed
\end{itemize}
\end{lemma}

The radical $\mu^1 \colon [-]_{\Gamma_1,0} \to 1_{\Ext_\Ec(\Cs)}$ corresponding  to $I_1$:
\[
\xymatrix{
0 \ar[r] & [-]_{\Gamma_1,0}  \ar[r]^-{\mu^1} & 1_{\Ext_\Ec(\Cs)}  \ar[r]^-{\eta^1} & H_1 I_1 \ar[r] & 0
}
\]
admits a nice description. First note that, for $f$ in  $\Ext_\Ec(\Cs)$, $\mu^1_f=(\overline{\mu}^1_f,0)$ where $\overline{\mu}^1_f$ is the kernel of $\overline{\eta}^1_f$. 
Then, since $\overline{\mu}^1_f=\pi_2 \circ  \overline{\mu}^1_{\pi_1}$ (see diagram \Ref{magicdiagram}),  it is easy to check that $  \overline{\mu}^1_f = \pi_2\circ \ker(\pi_1)\circ k$ where $k$ is the monomorphism in the commutative diagram
\[
\xymatrix{
[\Ker(f)]_{\Gamma,0} \ar[dd] \ar@/^2ex/[ddrr]^-{[\Ker(\pi_1)]_{\Gamma,0}} \ar@/_9ex/[dddd]_-{\mu_{\Ker(f)}}&&&&\\\\
[f]_{\Gamma,1} =[\pi_1]_{\Gamma,1}  \ar[rrdd]_-{\overline{\mu}^1_{\pi_1}}   \ar[dd]_-k \ar[rr]^-{\ker ([\pi_1]_{\Gamma,0})} \ar@{}[rrdd] && [\Eq(f)]_{\Gamma,0}  \ar[dd]^-{\mu_{\Eq(f)}}  \ar[rr]^-{[\pi_1]_{\Gamma,0}}  \ar[dd] && [A]_{\Gamma,0} \ar[dd]^-{\mu_{\eta_A}} \\\\
 \Ker(f) \ar[rr]_-{\ker(\pi_1)} && \Eq(f) \ar[rr]^-{\pi_1} \ar[dd]_-{\eta_{\Eq(f)}} \ar[rd]|-{\overline{\eta}^1_{\pi_1}} && A \ar[dd]^-{\eta_A} \\
 &&&T_1[\pi_1] \ar[ru]|-{T_1(\pi_1)}\ar[dl]
 \\
 &&HI(\Eq(f))\ar[rr]_-{HI(\pi_1)}   &&I(A)
}
\]
where the left hand square is a pullback. Note that if $f$ is normal, then $\Ker(f)$ is in $\Xs$ and the converse is true when $I$ is protoadditive. One also sees here that, for $f\colon A \to B$ in $\Split{\Ec}$, one has
\[
[f]_{\Gamma,1} = [A]_{\Gamma,0} \cap \Ker(f)
\]

\paragraph{Closure operators} Let $\mathcal{M}(\Ec)$ be the full subcategory of $\Arr(\Cs)$ determined by the kernels of morphisms in $\Ec$. For an object $k\colon K\to A  $ in $\mathcal{M}(\Ec)$, the closure of $k$ with respect to $\Xs$ is determined by the following rule:
\[
\overline{k}^\Xs_A= \ker(\eta_{\Coker(k)} \circ \coker(k)) = \coker(k)^{-1}(\mu_{\Coker(k)})
\]
This defines an endofunctor $\Pb^\Xs \colon\mathcal{M}(\Ec)\to \mathcal{M}(\Ec) $ which has the following properties:
\begin{itemize}
\item $\Cod =\Cod \circ \Pb^\Xs$;
\item $\forall k \in \mathcal{M}(\Ec) \colon k \leq \overline{k}^\Xs$;
\item $\forall k , l\in \mathcal{M}(\Ec) \colon  k \leq l \Rightarrow \overline{k}^\Xs \leq \overline{l}^\Xs$;
\item $\forall k \in \mathcal{M}(\Ec) \colon \overline{k}^\Xs =\overline{\overline{k}^\Xs}^\Xs$;
\end{itemize}
that is an \emph{idempotent closure operator} \cite{Tholen1}. This closure operator also has the following additional property:
$\forall f\colon A \to B \in \Ec$ and $k\colon K \to B \in \mathcal{M}(\Ec)$
\[
\overline{f^{-1}(k)}^\Xs = f^{-1}(\overline{k}^\Xs),
\]
that is every morphism in $\Ec$ is ``open''. In the context of a homological category $\Cs$ with $\Ec=\RegEpi(\Cs)$, such closure operators were called homological (see \cite{BG:Torsion}) and were already used to describe the first fundamental group functor in \cite{DuckertsEveraertGran}.

\subsection{Derived Galois structures (simple case)}

\begin{theorem}\label{firstderived} Let us assume that 
\begin{equation}\label{GStype1} 
\Gamma \colon \gal{(\Cs,\Ec)}{(I,H,\eta,\epsilon)}{(\Xs,\Fc)}
 \end{equation}
 is a closed Galois structure such that
 \begin{itemize}
 \item $\Cs$ is pointed protomodular;
 \item $(\Cs,\Ec)$ satisfies \hli{E4}, \hli{$E5$} and \hli{M};
 \item $I$ preserves pullbacks of type \Refi{special-pullback}.
 \end{itemize}
Then one has a closed Galois  structure
\[
\xymatrix{
\Gamma_1 \colon (\Ext_\Ec(\Cs),\Ec^1) \ar[rr]^-{(I_1,H_1,\eta^1,\epsilon^1)} &&   (\NExt_\Gamma(\Cs),\Fc^1)
}
\]
such that 
 \begin{itemize}
 \item $\Ext_\Ec(\Cs)$ is pointed protomodular;
 \item $(\Ext_\Ec(\Cs),\Ec^1)$ satisfies \hli{E4}, \hli{$E5$} and \hli{M};
 \item $I_1$ preserves pullbacks
 \[
\xymatrix{
a  \ar[r] \ar[d] & b \ar[d]^-g \\
d \ar[r]_-h & c
}
\]
with $h$ in $\Ec^1$ and $g$ in $\Split{\Ec^1}$.
 \end{itemize}
\end{theorem}
\begin{proof} 
We just have to check here that any pullback of the form above, i.e.\ any pullback
\[
\xymatrix@=15pt{
A_1 \ar@<-0.5ex>[dd] \ar[dr]_(0.6)a \ar[rr]  \ar@{}[drdr]|<<{\pullback} &&B_1 \ar[dr]^-b  \ar@<0.5ex>[dd]^(0.7){g_1}\\
&A_0 \ar@{}[dr]|<<{\pullback} \ar@<-0.5ex>[dd]  \ar[rr] &&B_0\ar@<0.5ex>[dd]^-{g_0} \\
D_1  \ar[rd]_-d \ar[rr]^(0.3){h_1} \ar@<-0.5ex>@{.>}[uu]&&C_1\ar[dr]^-c \ar@<0.5ex>@{.>}[uu]\\
&D_0\ar[rr]_-{h_0} \ar@<-0.5ex>@{.>}[uu] &&C_0\ar@<0.5ex>@{.>}[uu] .
}
\] 
with $h=(h_1,h_0)$ in $\Ec^1$ and $g=(g_0,g_1)$ in $\Split {\Ec^1}$, is preserved by $I_1$. If one takes kernel pairs, then by commutation of limits, one finds a diagram
\[
\xymatrix@=15pt{
\Eq(a) \ar@<-0.5ex>[dd]  \ar[dr]_-{\pi_1} \ar[rr] &&\Eq(b)  \ar[dr]^-{\pi_1}  \ar@<0.5ex>[dd]^(0.7){\hat g_1}\\
&A_1 \ar@{}[dr]|<<{\pullback} \ar@<-0.5ex>[dd]  \ar[rr] &&B_1\ar@<0.5ex>[dd]^-{g_1} \\
\Eq(d)  \ar[dr]_-{\pi_1}  \ar[rr]^(0.3){\hat h_1} \ar@<-0.5ex>@{.>}[uu]&&\Eq(c) \ar[dr]^-{\pi_1} \ar@<0.5ex>@{.>}[uu]\\
&D_1\ar[rr]_-{h_1} \ar@<-0.5ex>@{.>}[uu] &&C_1\ar@<0.5ex>@{.>}[uu] .
}
\]
where the front and back faces are pullbacks of the form (\ref{special-pullback}). Then applying the radical $\rf_\Gamma$ to this diagram one obtains a cube 
\[
\xymatrix@=15pt{
[ \Eq(a)]_{\Gamma,0} \ar@<-0.5ex>[dd] \ar[dr] \ar[rr] \ar@{}[dr]|<<{\pullback} &&[\Eq(b)]_{\Gamma,0}  \ar[dr] \ar@<0.5ex>[dd]^(0.7){[\hat g_1]_{\Gamma,0} }\\
&[A_1]_{\Gamma,0}  \ar@{}[dr]|<<{\pullback} \ar@<-0.5ex>[dd]  \ar[rr] &&[B_1]_{\Gamma,0} \ar@<0.5ex>[dd]^-{[g_1]_{\Gamma,0} } \\
[ \Eq(d)]_{\Gamma,0}   \ar[dr]  \ar[rr]^(0.3){[\hat h_1]_{\Gamma,0} } \ar@<-0.5ex>@{.>}[uu]&&[\Eq(c)]_{\Gamma,0}  \ar[dr] \ar@<0.5ex>@{.>}[uu]\\
&[ D_1]_{\Gamma,0} \ar[rr]_-{[h_1]_{\Gamma,0} } \ar@<-0.5ex>@{.>}[uu] &&[C_1]_{\Gamma,0} \ar@<0.5ex>@{.>}[uu].
}
\]
Consequently, one finds a pullback
\[
\xymatrix{
[a]_{\Gamma,1}  \ar[r]  \ar@<-0.5ex>[d]  \ar@{}[rd]|<<{\pullback} &[b]_{\Gamma,1}   \ar@<0.5ex>[d]\\
[d]_{\Gamma,1} \ar[r] \ar@<-0.5ex>@{.>}[u]& [c]_{\Gamma,1}  \ar@<0.5ex>@{.>}[u] 
}
\]
and one concludes using Lemma \ref{radref}.
\end{proof}

\subsection{Derived Galois structures (composite case)}

Let us consider a composite of closed Galois structures
\begin{equation}\label{GStypeC}
\xymatrix{
\Gamma'' \colon (\Cs,\Ec) \ar@^{->}[rr]^-{\Gamma=(I,H,\eta,\epsilon)}&&(\Xs,\Fc) \ar@^{->}[rr]^-{\Gamma'=(F,U,\theta,\zeta)}&&(\Ys,\Gc)
}
\end{equation}
where 
\begin{itemize}
\item $\Cs$ is pointed protomodular;
\item $(\Cs,\Ec)$ satisfies \hl{E4}, \hl{E5} and \hl{M};
\item $I$ preserves pullbacks of type \Ref{special-pullback};
\item $F$ is protoadditive.
\end{itemize}

\begin{theorem}\label{Caracterisation} Let $f: A \rightarrow B$ be in $\Ec$. The following conditions are equivalent:
\begin{enumerate}
 \item $f$ is an $\Gamma''$-normal extension;
 \item $f$ is $\Gamma$-normal and $\Ker (f) \in \Ys$.
\end{enumerate}
\end{theorem}

\begin{proof}
Let us assume that $f$ is a $\Gamma''$-normal extension. Then in the following commutative diagram the composite of the left pointing squares is a pullback: 
\[
\xymatrix{
FI(A\times_BA) \ar@<-0.5ex>[d]_{FI(\pi_1)} & I(A\times_BA) \ar@<-0.5ex>[d]_{I(\pi_1)} \ar[l] & E\times_B A \ar@{}[rd]|<<{\pullback}\ar@<-0.5ex>[d]_{\pi_1} \ar[l] \ar[r] & A \ar[d]^f\\
FI(A) \ar@{.>}@<-0.5ex>[u] & I(A) \ar[l]  \ar@{.>}@<-0.5ex>[u] & A \ar[l] \ar[r]_f  \ar@{.>}@<-0.5ex>[u] & B}
\]
Since the middle square is in $\Ec^1$, this implies that this square is, in fact, a pullback (see \cite[Lemma 1.1]{Gran:Central-Extensions}), and we find that $f$ is a $\Gamma$-normal extension. 
The fact that $\Ker (f)$ lies in $\Ys$ for any $\Gamma''$-normal extension $f$ has already been remarked above.

Let us assume now that $f$ is $\Gamma$-normal and $\Ker (f) \in \Ys$. Then one sees that in the commutative diagram 
\begin{center}
\begin{equation*}
\vcenter{
\xymatrix{
\Ker (f) \ar@{->}[rr]^-{\ker(\pi_1)}  \ar[d]_-{\eta_{\Ker(f)}} &&\Eq(f) \ar@<0.5ex>[r]^-{\pi_1}   \ar[d]_-{\eta_{\Eq(f)}}  \ar@{}[rd]|-{(3)} &A  \ar[d]^-{\eta_A}  \ar@<0.5ex>@{.>}[l]\\
I(\Ker(f)) \ar[rr]^-{I(\ker(\pi_1))}  \ar[d]_{\theta_{I(\Ker(f))}} &&I(\Eq(f)) \ar@<0.5ex>[r]^-{I(\pi_1)} \ar[d]_-{\theta_{I(\Eq(f))}}  \ar@{}[rd]|-{(4)} &I(A) \ar[d]^-{\theta_{I(A)}}  \ar@<0.5ex>@{.>}[l] \\
FI(\Ker (f)) \ar[rr]_-{FI(\ker(\pi_1))} &&FI(\Eq(f)) \ar@<0.5ex>[r]^-{FI(\pi_1)}   &FI(A) \ar@<0.5ex>@{.>}[l]} 
}
\end{equation*}
\end{center}
the square $(3)$ is a pullback (since $f$ is $\Gamma$-normal) and $\eta_{\Ker(f)}$ an isomorphism. It follows that the second row is a split short exact sequence. By protoadditivity of $F$, the third row is  also a split short exact sequence, and we obtain by protoadditivity of $\Cs$ that $(4)$ is a pullback because $\theta_{I(\Ker(f))}$ is an isomorphism, by assumption. Thus $(3)$+$(4)$ is a pullback and $f$ is a $\Gamma''$-normal extension.
\end{proof}

\begin{theorem}\label{secondderived}
One has a composite of closed Galois structures
\[
\xymatrix{
\Gamma''_1 \colon (\Ext_\Ec(\Cs),\Ec^1) \ar@^{->}[rr]^-{(I_1,H_1,\eta^1,\epsilon^1)}&&(\NExt_\Gamma(\Cs),\Fc^1) \ar@^{->}[rr]^-{(F_1,U_1,\theta^1,\zeta^1)}&&(\NExt_{\Gamma''}{}(\Cs),\Gc^1)
}
\]
of type \Refi{GStypeC} where $F_1= (F\circ I)_1\circ H_1$ is protoadditive.
\end{theorem}

\begin{proof} it suffices to check that $F_1$ is protoadditive. This follows from the next lemma.
\end{proof}

 \begin{lemma}\label{Shape-protoadditive-radical} Let $f \colon A \rightarrow B$ be a $\Gamma$-normal extension. Then 
 \[
[f]_{\Gamma'',1} =\overline{0}_{\Ker (f)}^{\Fc}=[\Ker(f)]_{\Gamma'',0} .
 \]
 This implies that $[-]_{\Gamma_1'',1}$ restricted to $\NExt_\Gamma(\Cs)$ is protoadditive and therefore, that the restriction $F_1$ of $(F\circ I)_1$ to $\NExt_\Gamma(\Cs)$ is also protoadditive.
  \end{lemma}
 \begin{proof}
First note that $\Ker(f)$ lies in $\Xs$ since $f$ is a $\Gamma$-normal extension. Then, as in the proof of Theorem \ref{Caracterisation}, we see that the last row in the following diagram is exact:
\[
\xymatrix{
[f]_{\Gamma'',1} \ar[rr]^-{\ker( [\pi_1]_{\Gamma'',0})} \ar@{->}[d]_-{k} && [\Eq(f)]_{\Gamma'',0} \ar@<0.5ex>[rr]^-{ [\pi_1]_{\Gamma'',0}} \ar[d] &&[A]_{\Gamma'',0}\ar[d]^-{\mu_A} \ar@<0.5ex>@{.>}[ll] \\
\Ker (f)  \ar[d]_-{\eta_{\Ker (f)}} \ar[rr]_-{\ker(\pi_1)}&& \Eq(f) \ar@<0.5ex>[rr]^-{\pi_1}   \ar[d] && A \ar@<0.5ex>@{.>}[ll] \ar[d] \\
FI(\Ker (f)) \ar[rr]_-{FI(\ker(\pi_1))} && FI(\Eq(f)) \ar@<0.5ex>[rr]^-{FI(\pi_1)} && FI(A)  \ar@<0.5ex>@{.>}[ll].
}
\]
It follows that the left column in this diagram is also exact (note that the right solid bottom square is a double $\Ec$-extension) and the first part of the result follows. Now, let us start with a split short exact sequence 
\[
\xymatrix{
0 \ar[r]&K_1 \ar[r]^-{\ker(p_1)}  \ar[d]_-k &A_1 \ar@<0.5ex>[r]^-{p_1} \ar[d]_-a&B_1 \ar[d]^-b  \ar@<0.5ex>@{.>}[l]^{s_1} \ar[r]&0\\
0 \ar[r]&K_0 \ar[r]_-{\ker(p_0)}&A_0 \ar@<0.5ex>[r]^{p_0}&B_0 \ar@<0.5ex>@{.>}[l]^-{s_0}  \ar[r]&0
}
\]
in $\NExt_\Gamma(\Cs)$. Then, by commutation of limits, one finds that the first row in the commutative diagram
\[
\xymatrix{
\Ker(k) \ar[d] \ar[r]&\Ker(a)\ar[d] \ar@<0.5ex>[r]&\Ker(b)\ar[d] \ar@<0.5ex>@{.>}[l]\\
K_1 \ar[r]^-{\ker(p_1)}  \ar[d]_-k &A_1 \ar@<0.5ex>[r]^-{p_1} \ar[d]_-a&B_1 \ar[d]^-b  \ar@<0.5ex>@{.>}[l]^{s_1}\\
K_0 \ar[r]_-{\ker(p_0)} &A_0 \ar@<0.5ex>[r]^{p_0} &B_0 \ar@<0.5ex>@{.>}[l]^-{s_0}
}
\]
is a split short exact sequence and, since $[-]_{\Gamma'',0}$ restricted to $\Xs$ is protoadditive, that one has a split short exact sequence
\[
\xymatrix{
0 \ar[r]& [k]_{\Gamma'',1} \ar[r] & [a]_{\Gamma'',1} \ar@<0.5ex>[r] & [b]_{\Gamma'',1}  \ar@<0.5ex>@{.>}[l]  \ar[r]&0.
}
\]
\end{proof} 

\subsection{Higher extensions and presentations}

In the sequel, we will use the notation for the finite ordinals: $0=\emptyset$ and $n=\{0,\hdots, n-1\}$ for $n\geq 1$. We now adopt many notations from \cite{EGVdL}. We write $\mathscr{P}(n)$ for the poset of subsets of $n$ viewed as a category. Let $(\Cs,\Ec)$ be a pair that satisfies the axioms \hl{E1} to \hl{E3}. The category $\Ext^n_\Ec(\Cs)$ is the full subcategory of $\Cs^{(\mathscr{P}(n)^{\mathsf{op}})}$ determined by the $n$-fold $\Ec$-extensions.  An \emph{$n$-fold $\Ec$-extension $A$} is a functor $A \colon \mathscr{P}(n)^{\mathsf{op}}\to \Cs$ such that for all $0\neq I \subseteq n$, the limit $\lim_{J\subsetneq I} A(J)$ exists and the induced morphism $A(I) \to \lim_{J\subsetneq I} A(J)$ is in $\Ec$. We shall use the notation $A_S=A(S)$, $a_{S}^{T}=A(S,T)\colon A(T) \to A(S)$ (for $S\subseteq T\subseteq n$),  $a_i=a_{n\setminus \{i\}}^n=A(n\setminus \{i\},n)$ and $A=(A_S)_{S\subseteq n}$. Let us define, for $i\in \Nbb$, $s_i \colon \mathbb{N} \to \mathbb{N}$ by
\[
s_i(k)=\left\lbrace \begin{array}{ll}\text{$k$}&\text{if $k<i$}\\
\text{$k+1$} &\text{if $k\geq i$}
 \end{array}\right.
 \]
 and $-^i\colon \mathscr{P}(\mathbb{N}) \to \mathscr{P}(\mathbb{N})$ by $S^i=\{ s_i(k)\, | \,k\in S \}$. 
For every $n\geq 0$, one has a pair $(\Ext_\Ec^n(\Cs),\Ec^n )$ which satisfies the axioms  \hl{E1} to \hl{E3}.  For any $0\leq i < n$, there is also an isomorphism
\[
\delta_i \colon  (\Ext_\Ec^n(\Cs),\Ec^n) \to (\Ext_{\Ec^{n-1}} (\Ext^{n-1}_\Ec(\Cs)) , (\Ec^{n-1})^1)
\]
which maps a $n$-extension $A$ to the natural transformation
\[
\delta_i(A)=(a_S^{S^i\cup\{i\}})_{S\subseteq n-1} \colon (A_{S^i\cup \{i\}})_{S\subseteq n-1} \to (A_{S^i})_{S\subseteq n-1}
\] 
and an arrow $f\colon A\to B$ between $n$-extensions to:
\[
\xymatrix{
(A_{S^i\cup \{i\}})_{S\subseteq n-1} \ar[rr]^-{(f_{S^i\cup\{i\}})_{S\subseteq n-1}} \ar[d] &&(B_{S^i\cup \{i\}})_{S\subseteq n-1} \ar[d]\\
(A_{S^i})_{S\subseteq n-1} \ar[rr]_-{(f_{S^i})_{S\subseteq n-1}} &&(B_{S^i})_{S\subseteq n-1}}
\]
Actually, the class $\Ec^n$ is defined inductively by transport along the isomorphism $\delta_{n-1}$ and it is only proved afterwards that, for $i=0,\hdots,n-2$, $\delta_i$ also preserves extensions (see \cite[Proposition 1.16]{EGoeVdL} for more details). Here we have made the identifications  $\Ext^0_\Ec(\Cs) = \Cs$, $\Ext^1_\Ec(\Cs)=\Ext_\Ec(\Cs)$ and the corresponding identifications of classes of extensions.
We define also $\delta_i^+= \Cod \circ \delta_i$ and $\delta_i^-=\Dom \circ \delta_i$. Here the functors $\Cod$ and $\Dom$  are the codomain and domain functors (they send an arrow to its codomain and domain, respectively).
\begin{lemma}\label{carredeltaij} For $A$ in $\Ext^n_\Ec(\Cs)$ and $0\leq i<j < n$, the diagram in $\Ext^{n-2}_\Ec(\Cs)$
\[
 \xymatrix{
. \ar[r]^-{\delta_{j-1}(\delta^-_i(A))} \ar[d]_-{\delta_{i}(\delta^-_j(A))}&. \ar[d]^-{\delta_{i}(\delta^+_j(A))} \\
.\ar[r]_-{\delta_{j-1}(\delta^+_i(A))}&.
}
\]
is commutative.
\end{lemma}
Let us also recall that for every $0\leq i <n$, one has an isomorphism
\[
\rho_i \colon (\Ext^n_\Ec(\Cs),\Ec^n) \to (\Ext^{n-1}_{\Ec^1}(\Ext_\Ec(\Cs)) ,(\Ec^1)^{n-1})
\]
which maps a $n$-fold $\Ec$-extension $A$ to
\[
\rho_i(A)=(a_{S^i}^{S^i\cup \{ i\}} \colon A_{S^i\cup \{ i\}} \to A_{S^i})_{S\subseteq n-1}.
\]
Using this, we can also define the isomorphism
\[
(\rho_i,\rho_i) \colon \Ext_{\Ec^n} (\Ext^n_\Ec(\Cs)) \to  \Ext_{(\Ec^1)^{n-1}}(\Ext_{\Ec^1}^{n-1}(\Ext_\Ec(\Cs)))
\] 
which is determined on objects by 
\[
(\rho_i,\rho_i) (f \colon A \to B) =  \rho_i(f) \colon \rho_i(A) \to  \rho_i(B)
\]
The following rules hold (see \cite[Lemma 4.2]{EverHopf}): for $i<j$
\[
\delta_{j-1} \circ \rho_i= (\rho_i,\rho_i) \circ \delta_j  
\]
and for $j<i$
\[
\delta_j\circ \rho_i= (\rho_{i-1},\rho_{i-1})\circ \delta_j.
\]
Let us recall from  \cite{RVdL}  the following definition. Let $n\geq 1$. The \emph{direction} of a $n$-fold $\Ec$-extension $A$ is 
\[
\Ker^n (A)=\bigcap_{i=0}^{n-1} \Ker (a_i).
\]
This defines a functor 
\[
\Ker^n \colon \Ext^n_\Ec(\Cs) \to \Cs
\]
such that 
\[
\Ker^n (A)\cong \Ker (\Ker^{n-1} (\rho_i(A))) \cong \Ker^{n-1}(\Ker (\delta_i(A)))
\]
for every $i=0,\hdots,n-1$. We shall also consider the functor 
\[
\iota^n\colon \Cs \to  \Ext_\Ec^n(\Cs)
\]
that maps an object $A$ in $\Cs$ to the $n$-fold $\Ec$-extension given by 
\[
(\iota^n(A))_S=\left\lbrace \begin{array}{ll}  A &\text{if $S = n$}\\ 0 &\text{if $S \neq n$}
 \end{array}\right.
 \]
The left adjoint of $\iota^n$ is $\Dom^n \colon \Ext_\Ec^n(\Cs)\to \Cs $ given by $\Dom^n(A)=A_n$ and the right adjoint of $\iota^n$ is $\Ker^n$. Note that $\Ker^n(\iota^n(A))\cong A =\Dom^n (\iota^n(A))$ and that 
\[
\Dom^n (A)\cong \Dom (\Dom^{n-1} (\rho_i(A))) \cong \Dom^{n-1}(\Dom(\delta_i(A))).
\]
There also exists a functor $\Cod^n\colon \Ext_\Ec^n(\Cs) \to \Cs$ given by $\Cod^n(A)=A_0$ with properties similar to the properties of the functor $\Dom^n$.

\paragraph{Projective objects and presentations} 
Let $(\Cs,\Ec)$ satisfies \hl{E1} to \hl{E3}. An object $P$ of $\Cs$ is \emph{$\Ec$-projective} if for any morphism $f\colon A \to B $ in $\Ec$ the function
\[
\mathrm{Hom}_\Cs(P,f)\colon \mathrm{Hom}_\Cs(P,A)\to\mathrm{Hom}_\Cs(P,B)
\]
is surjective. One says that $\Cs$ has enough $\Ec$-projective objects if every object $C$ of $\Cs$ has at least a \emph{1-fold $\Ec$-projective presentation}, i.e.\ there exists a morphism $f\colon P \to C$ in $\Ec$ with an $\Ec$-projective domain. Let us recall that, if $\Cs$ has enough $\Ec$-projective objects, $\Ext^n_\Ec(\Cs)$ has enough $\Ec^n$-projective objects and an object $P$ in $\Ext^n_\Ec(\Cs)$ is $\Ec^n$-projective if and only if  $P_S$ is $\Ec$-projective for every $S\subseteq n$. For $C$  in $\Cs$ and $P$ in $\Ext_\Ec^n(\Cs)$ ($n\geq 2$), one says that $P$ is a \emph{$n$-fold $\Ec$-projective presentation} of $C$ if $P_S$ is $\Ec$-projective for every $S\neq 0$ and $\Cod^n(P)=P_0=C$. One denotes the category of $n$-fold $\Ec$-projective presentations by $\Pres_\Ec^n(\Cs)$. For $C$ in $\Cs$, we write $\Pres_\Ec^n(C)$ for the fibre over $C$ of the functor $\Cod^n\colon \Pres_\Ec^n(\Cs) \to \Cs$. Note that there is at least one morphism between any two objects of $\Pres_\Ec^n(C)$.

  \begin{proposition}\label{enough-projectives}
   Let  $(F,U,\eta,\epsilon): \Cs \rightharpoonup \Xs$ be an adjunction, 
   \[
   \Gbb=(G=F\circ U, \epsilon\colon G\to 1_\Cs, \delta \colon G \to G^2)
   \]
   the induced comonad on $\Cs$, and $\Pc$ the class of morphisms in $\Cs$ which are sent by $U$ to split epimorphisms in $\Cs$. Then $\Cs$ has enough $\Pc$-projective objects and $(\Cs,\Pc)$ satisfies \hli{E1} to \hli{E3}.\qed
 \end{proposition}
 
 This proposition can be applied to the following categories (for some semi-abelian theory $\Tbb$): $\Tbb(\Set)$ and $\Tbb(\HComp)$ which are monadic over $\Set$ (see \cite{Manes} for the monadicity of the second one) and $\Tbb(\Top)$  which is monadic over $\Top$ (\cite{Borceux-Clementino},\cite{Wyler}). In each of theses cases, the class $\Pc$ sits inside the class of regular epimorphisms (note that a regular epimorphism in $\Tbb(\Top)$ is a surjective (open) homomorphism whose codomain has the quotient topology \cite{Borceux-Clementino}).

\subsection{Higher derived Galois structures}

Using our description of the radical $\rf_{\Gamma_1}$ and using the same kind of principles  as in \cite{EverHopf}, one finds by induction and Theorem \ref{firstderived} the following

\begin{theorem} Given $\Gamma$ of type \ref{GStype1}, one has for all $n\geq 1$ a closed Galois structure 
\[
\gal{\Gamma_n\colon  (\Ext^n_\Ec(\Cs),\Ec^n)}{(I_n,H_n,\eta^n,\epsilon^n)}{(\NExt^n_\Gamma(\Cs),\Fc^n)}
\]
of type \Refi{GStype1} where $I_n$ is defined as the unique functor such that the diagram
\[
\xymatrix{
\Ext^n_\Ec(\Cs) \ar[r]^-{I_n} \ar[r] \ar[d]_-{\delta_i}& \NExt^n_\Gamma(\Cs) \ar[d]^-{\delta_i}\\
  \Ext_{\Ec^{n-1}}(\Ext^{n-1}_\Ec(\Cs))  \ar[r]_-{(I_{n-1})_1} &\NExt_{\Gamma_{n-1}}(\Ext^{n-1}_\Ec(\Cs))
}
\]
commutes for $i=0,\hdots,n-1$. The radical $[-]_{\Gamma_n,0}$ factors as $\iota^n([-]_{\Gamma,n})$ for some functor $[-]_{\Gamma,n}\colon \Ext^n_\Ec(\Cs) \to \Xs$ and one has  $[\delta_i(-)]_{\Gamma_{n-1},1}=\iota^{n-1}[-]_{\Gamma,n}$. The diagram 
\[
\xymatrix{
\Ext^n_\Ec(\Cs) \ar[r]^-{I_n} \ar[r] \ar[d]_-{\rho_i}& \NExt^n_\Gamma(\Cs) \ar[d]^-{\rho_i}\\
  \Ext_{\Ec^1}^{n-1}(\Ext_\Ec(\Cs))  \ar[r]_-{(I_{1})_{n-1}} &\NExt_{\Gamma_1}^{n-1}(\Ext_\Ec(\Cs))
}
\]
also commutes and $[\rho_i(-)]_{\Gamma_1,n-1}=\iota^1 [-]_{\Gamma,n}$  for $i=0,\hdots,n-1$. 

\end{theorem}

Similarly to the previous theorem one has

\begin{theorem}
Given a Galois structure $\Gamma''$ of type \ref{GStypeC}, one has for every $n\geq 1$ a composite of closed Galois structures
\[
\xymatrix{
\Gamma''_n \colon (\Ext^n_\Ec(\Cs),\Ec^n) \ar@^{->}[rr]^-{(I_n,H_n,\eta^n,\epsilon^n)}&&(\NExt_\Gamma^n(\Cs),\Fc^n) \ar@^{->}[rr]^-{(F_n,U_n,\theta^n,\zeta^n)}&&(\NExt_{\Gamma''}^n{}(\Cs),\Gc^n)
}
\]
of type \Refi{GStypeC} where $F_n= (F\circ I)_n\circ H_n$ is protoadditive.
\end{theorem}

\section{Fundamental group functors}
\subsection{Galois groups}
Let us assume that 
\begin{equation}\label{GStypeP}
\Gamma \colon \gal{(\Cs,\Ec)}{(I,H,\eta,\epsilon)}{(\Xs,\Fc)}
 \end{equation}
 is a closed Galois structure such that
 \begin{itemize}
 \item $\Cs$ is pointed, protomodular, and admits intersections 
\[
[\Dom(p)]_{\Gamma,0} \cap \Ker(p)
\]
for any morphism $p\colon E \to B$ in $\Ec$;
 \item $(\Cs,\Ec)$ satisfies \hl{E4}, \hl{E5} and \hl{M};
 \item $I$ preserves pullbacks of type \Ref{special-pullback}.
 \end{itemize} 
\paragraph{Galois groupoid} The \emph{Galois groupoid}  \cite{Janelidze:Hopf} $\Gal_\Gamma(p)$ of a normal extension $p$  is the image under $I$ of the kernel pair of $p$ (viewed as an internal groupoid in $\Cs$)
\[
\xymatrix{
\Eq(p) \times_E \Eq(p) \ar@<1ex>[rr]^-{p_1}  \ar@<-1ex>[rr]_-{p_2}  \ar[rr]|-{\tau} &&\Eq(p) \ar@<1ex>[rr]^-{\pi_1} \ar@<-1ex>[rr]_-{\pi_2} \ar@(ul,ur)[]^-{\sigma} && E\ar@{>}[ll]|-{\delta}.
}
\]
That is, $\Gal_\Gamma(p)$ is the internal groupoid in $\Xs$
\[
\xymatrix{
I(\Eq(p) \times_E \Eq(p)) \ar@<1ex>[rr]^-{I(p_1)}  \ar@<-1ex>[rr]_-{I(p_2)} \ar[rr]|-{I(\tau)}&& I(\Eq(p)) \ar@<1ex>[rr]^-{I(\pi_1)} \ar@<-1ex>[rr]_-{I(\pi_2)} \ar@(ul,ur)[]^-{I(\sigma)} & &I(E)\ar@{>}[ll]|-{I(\delta)}
}
\]
\paragraph{Galois-group} The \emph{Galois group} \cite{Janelidze:Hopf}  of a normal extension $p\colon E\to B$ is defined as the object $\Gal_\Gamma (p,0)$ in the following pullback:
\[
\xymatrix{
\Gal_\Gamma (p,0) 
\ar@{}[rd]|<<<{\pullback} \ar[r] \ar[d]&HI(\Eq(p)) \ar[d]^-{\langle  HI(\pi_1),HI(\pi_2) \rangle}\\
0\ar[r]& HI(E)\times HI(E).
}
\]
The Galois group construction gives in fact a functor
\[
\Gal_\Gamma(-,0)\colon \NExt_\Gamma(\Cs) \to \Ab(\Xs)
\]
which is a Baer invariant with respect to the functor $\Cod \colon \NExt_\Gamma (\Cs)\to \Cs$, i.e.\:
\begin{proposition}\label{Gal1-Baer} \cite{EGoeVdL2} Two morphisms 
\[
f=(f_1,f_0),g=(g_1,g_0)\colon p \to q
\]
in $\NExt_\Gamma(\Cs)$
such that $f_0=\Cod (f) = \Cod(g)=g_0$ have the same image by $\Gal_\Gamma (-,0)$
\begin{equation*}
\begin{split}
\Gal_\Gamma(f,0)=\Gal_\Gamma(g,0)\colon \Gal_\Gamma(p,0)\to \Gal_\Gamma(q,0). 
\end{split}
\end{equation*}\qed
\end{proposition}
Since $\Cod$ commutes with $I_1$, one finds
\begin{corollary}\label{corobaer} The functor
\[
\Gal_\Gamma( I_1(-),0) \colon \Ext_\Ec(\Cs) \to \Ab(\Xs)
\]
is a Baer invariant with respect to $\Cod$.
\end{corollary}
\begin{theorem}\label{description1-pi1}  \cite{Janelidze:Hopf} One has a natural isomorphism
\[
\Gal_\Gamma(-,0)\cong [\Dom(-)]_{\Gamma,0} \cap  \Ker(-).
\]
\end{theorem}
The Galois group $\Gal_\Gamma(p,0)$ inherits an internal group structure from the composition of the groupoid $\Gal_\Gamma(p)$ and can be viewed, internally, as the group of automorphisms of $0$. But  as we have recalled, if $\Cs$ is a finitely complete protomodular category, any of its objects underlies at most one internal group structure so that we do not have to keep track of the group structure: only the object itself will be of interest.

 \subsection{The fundamental group functors as Kan extensions}

We now consider a Galois structure
\begin{equation}\label{GStypePP}
\Gamma \colon \gal{(\Cs,\Ec)}{(I,H,\eta,\epsilon)}{(\Xs,\Fc)}
 \end{equation}
 of type \Ref{GStypeP} with a pair $(\Cs_\Pc,\Pc)$ such that \hl{E1} to \hl{E3} hold and such that $\Cs_\Pc$ is a subcategory of $\Cs$ with sufficiently projective objects with respect to  a subclass $\Pc$ of $\Ec$. Then $\Ext_\Pc^n(\Cs_\Pc)$ is a subcategory of $\Ext_\Ec^n(\Cs)$.

For $n\geq 1$ one defines a functor $\Gs^{\Gamma,\Pc}_{n} \colon \Ext^n_\Pc(\Cs_\Pc) \to \Ab(\Xs)$ by
\[
\Gs^{\Gamma,\Pc}_{n} = \Dom^{n-1}\left( \Gal_{\Gamma_{n-1}}(\delta_{n-1} I_n(-),0)\right)
\]
Note that this functor is well defined. Indeed, for $P$ a $n$-extension, 
\[
\Gal_{\Gamma_{n-1}}(\delta_{n-1} I_n(P),0) = \iota^{n-1}(A)  \in \NExt^{n-1}_\Gamma(\Cs)
\]
for some $A$ in $\Cs$, and one finds $A = \Dom^{n-1} ( \iota^{n-1}(A)) =\Ker^n( \iota^{n-1}(A)) \in \Xs$. One could also use $\Ker^{n-1}$ instead of   $\Dom^{n-1}$ in the definition of $\Gs^{\Gamma,\Pc}_{n}$. Consequently, the internal abelian group structure is also preserved ($\Ker^{n-1}$ preserves limits). Moreover, if one consider $\Gal_{\Gamma_{n-1}}(-,0)$ as a functor of type $\NExt_{\Gamma_{n-1}}(\Ext_\Ec^{n-1}(\Cs))  \to \NExt_\Gamma^{n-1}(\Cs)$, one obtains, using Theorem \ref{description1-pi1}, a mono-morphism 
\[
 \gamma \colon \Gal_{\Gamma_{n-1}}(-,0) \Rightarrow \Ker 
\]
and consequently another monomorphism
\[
\overline {\gamma} =\Ker^{n-1} \gamma \delta_{n-1} I_n  \colon \Gs_{n}^{\Gamma,\Pc}  \Rightarrow \Ker^n\circ I_n.
\]
whose component at $P$ is $\Ker^{n-1}( \ker( \eta^{n-1}_{\delta_{n-1}^- I_n (P)} \circ \ker(\delta_{n-1} I_n(P))))$.

\begin{lemma}\label{invargal2} The functor $\Gs^{\Gamma,\Pc}_{n} $ is an invariant with respect to $\Cod \circ \delta_i=\delta_i^+$ for $0 \leq i <n$.

\end{lemma}

For this we use the following 

\begin{lemma}\label{invargal} For $P$ any $n$-fold $\Pc$-projective presentation
\[
\Gs_n^{\Gamma,\Pc}(P) = \Dom^{n-1}\left(  \Gs_{1}^{\Gamma_{n-1},\Pc^{n-1}}( \delta_i P)\right)
\]
for $0\leq i <n$.
\end{lemma}
\begin{proof} First, let us note that
\[
\Gs_{1}^{\Gamma_{n-1},\Pc^{n-1}}( \delta_i P) = \Gal_{(\Gamma_{n-1})_0}( (I_1)_{n-1}\delta_i (P),0) = \Gal_{\Gamma_{n-1}}( \delta_i I_n (P),0)
\]
So we are going to prove that
\[
\Dom^{n-1}  \Gal_{\Gamma_{n-1}}(\delta_{i} I_n(P),0) =\Dom^{n-1}  \Gal_{\Gamma_{n-1}}(\delta_{j} I_n(P),0).
\]
for $0\leq i <j<n$. Now, let us remark that, since $\delta^-_i(P)$ and $\delta^-_j(P)$ are both projective, both $\delta_{j-1}(\delta^-_i(P))$ and $\delta_i(\delta^-_j(P))$ are in $\Split {\Ec^{n-2}}$. Consequently, one finds a commutative diagram
\[
 \xymatrix{
. \ar[rrrr]^-{\delta_{j-1}\delta^-_i(P)} \ar[ddd]_-{\delta_{i}\delta^-_j(P)} \ar[rd] &&&&. \ar[ddd]^-{\delta_{i}\delta^+_j(P) = \delta_{i}\delta^+_j I_n(P)} \\
&.\ar[ldd]^-{\delta_i \delta_j^- I_n(P)} \ar[rrru]_-{\delta_{j-1} \delta_i^- I_n (P)}
\\\\
.\ar[rrrr]_-{\delta_{j-1}\delta^+_i(P) =\delta_{j-1}\delta^+_i I_n(P)}&&&&.
}
\]
where  $\delta_i \delta_j^- I_n(P)$ and $\delta_{j-1} \delta_i^- I_n (P)$ are both in $\Split {\Ec^{n-2}}$ and 
\begin{equation*}
\begin{split}
& [\delta_{j-1} \delta_i^-I_n (P)]_{\Gamma_{n-2},1}   \cap \Ker(\delta_i \delta_j^-I_n(P))\\
=& [\delta^-_{j-1} \delta_i^-I_n (P)]_{\Gamma_{n-2},0} \cap \Ker(\delta_{j-1} \delta_i^-I_n (P))  \cap \Ker(\delta_i \delta_j^-I_n(P)) \\
=& [\delta^-_{i} \delta_j^-I_n (P)]_{\Gamma_{n-2},0} \cap \Ker(\delta_{i} \delta_j^-I_n (P))  \cap \Ker(\delta_{j-1} \delta_i^-I_n (P))\\
=& [  \delta_{i} \delta_j^-I_n (P)]_{\Gamma_{n-2},1} \cap \Ker(\delta_{j-1} \delta_i^-I_n (P))
\end{split}
\end{equation*}
Therefore 
\begin{equation*}
\begin{split}
&   \delta_{j-1} \left(  [\delta_i^-I_n (P) ]_{\Gamma_{n-1},0} \cap \Ker( \delta_iI_n (P))   \right) \\
=&  \delta_{j-1} [\delta_i^-I_n (P)]_{\Gamma_{n-1},0} \cap \delta_{j-1} \Ker( \delta_iI_n (P))\\
=&\iota^1 \left(  [\delta_{j-1} \delta_i^-I_n (P)]_{\Gamma_{n-2},1}   \right) \cap \Ker(\delta_{j-1} \delta_iI_n (P)) \\
= & \iota^1 \left([\delta_{j-1} \delta_i^-I_n (P)]_{\Gamma_{n-2},1}   \cap \Ker(\delta_i \delta_j^-I_n(P)) \right)\\
=& \iota^1 \left( [  \delta_{i} \delta_j^-I_n (P)]_{\Gamma_{n-2},1} \cap \Ker(\delta_{j-1} \delta_i^-I_n (P)) \right)\\
=&\iota^1 \left([\delta_{i} \delta_j^-I_n (P)]_{\Gamma_{n-2},1} \right)\cap \Ker(\delta_i \delta_jI_n(P)) \\
=&   \delta_{i} [\delta_j^-I_n (P)]_{\Gamma_{n-1},0} \cap \delta_{i} \Ker( \delta_jI_n (P))     \\
=& \delta_i \left([\delta_j^-I_n (P)]_{\Gamma_{n-1},0}  \cap \Ker(\delta_j I_n (P))\right) 
\end{split}
\end{equation*}
Finally, applying the composite $\Dom^{n-2} \circ \Dom $ to both sides gives the result.
\end{proof}

\begin{lemma}\label{invargal3} For $P$ any $n$-fold $\Pc$-projective presentation
\[
\Gs_n^{\Gamma,\Pc}(P) = \Dom  \left( \Gs_{n-1}^{\Gamma_1, \Pc^1}(\rho_i(P)) \right)
\]
for $0 \leq i <n$.
\end{lemma}
\begin{proof} Let $0 \leq i <n-1$. First let us remark that 
\begin{equation*}
\begin{split}
 \Gs_{n-1}^{\Gamma_1, \Pc^1}(\rho_i(P)) &= \Dom^{n-2} \left( \Gal_{(\Gamma_1)_{n-2}} (\delta_{n-2} (I_1)_{n-1} \rho_i P, 0) \right)\\
 &= \Dom^{n-2} \left(  \Gal_{(\Gamma_1)_{n-2}} (  (\rho_i, \rho_i)  \delta_{n-1} I_n (P), 0)   \right)
 \end{split}
\end{equation*}
Now $(\rho_i, \rho_i)  \delta_{n-1} I_n (P)$ is the $\Ec^n$-extension 
\[
\rho_i( \delta_{n-1} I_n (P)) \colon \rho_i( \delta_{n-1}^- I_n (P)) \to \rho_i( \delta_{n-1}^+ I_n (P))
\]
and $\Gal_{(\Gamma_1)_{n-2}} (  (\rho_i, \rho_i)  \delta_{n-1} I_n (P), 0)$  is given by the pullback
\[
\xymatrix{
\Gal_{(\Gamma_1)_{n-2}} (  (\rho_i, \rho_i)  \delta_{n-1} I_n (P), 0) \ar[r] \ar[d]&[\rho_i( \delta_{n-1}^- I_n (P)) ]_{(\Gamma_1)_{n-2},0}  \ar[d]\\
\Ker (\rho_i( \delta_{n-1} I_n (P))) \ar[r]&\rho_i( \delta_{n-1}^- I_n (P)) 
}
\]
or equivalently by the pullback
\[
\xymatrix{
\rho_i (\Gal_{\Gamma_{n-1}} (\delta_{n-1} I_n (P), 0)) \ar[r] \ar[d]&   \rho_i  (  [\delta_{n-1}^- I_n (P)) ]_{\Gamma_{n-1},0} ) \ar[d]\\
\rho_i (\Ker( \delta_{n-1} I_n (P))) \ar[r]&\rho_i( \delta_{n-1}^- I_n (P)) 
}
\]
The case $i=n-1$ is proved in a similar way with the help of the previous lemma.
\end{proof}

\begin{definition} Let $n\geq 1$. The \emph{$n^{\text{th}}$-fundamental group functor} $\pi_n^{\Gamma,\Pc}(-)$, with respect to the Galois structure $\Gamma$ and the class $\Pc$, is the pointwise right Kan extension of $\Gs_{n}^{\Gamma,\Pc}=\Dom^{n-1}\left( \Gal_{\Gamma_{n-1}}(I_n(-),0)\right)$ along the functor $\Cod^n$. 
\[
\xymatrix{
&\Ext^n_\Pc(\Cs_\Pc) \ar[dl]_-{\Cod^n} \ar[dr]^-{\Gs_{n}^{\Gamma,\Pc}}& \\
\Cs_\Pc \ar@{.>}[rr]_-{\pi_n^{\Gamma,\Pc}(-)} \ar@{}[urr]|(0.4){\Rightarrow}&& \Gp(\Xs)
}
\]
\end{definition}

Indeed, we need to prove that the fundamental group functors exists.

\begin{lemma}\label{cofinal} Let $n\geq 0$, $C$ be in $\Cs_\Pc$ and let us consider the comma square  
\[
\xymatrix{
C\downarrow \Cod^n \ar[r]^-{Q_C}   \ar[d]_-{P_C}  & \ar[d]^-{\Cod^n}  \Ext^n_\Pc(\Cs_\Pc) \\
   {\bf 1}  \ar[r]_-C  \ar@{}[ru]|-{\Rightarrow} & \Cs_\Pc
}
\]
The full subcategory $\Ps_C^n$ of $C\downarrow \Cod^n$, whose objects are of type
\[
(1_C \colon C \to C= \Cod^n(P),P)
\]
 with $P$ in $\Pres^n_\Pc(C)$, is initial. In fact, for every object
\[
\Qc =(h \colon C \to \Cod^n (Q),Q)
\]
of $C\downarrow \Cod^n$, there exists a weak terminal object in $ \Ps_C^n\downarrow \Qc$.
\end{lemma}

\begin{proof} Let $n\geq 0$ ang $\Qc =(h \colon C \to \Cod^n (Q),Q)$ be in $C\downarrow \Cod^n$. Recall that an object in $\Ps_C^n\downarrow \Qc$ ($n\geq 0$) is a map
\[
\xymatrix{
(1_C \colon C \to C= \Cod^n(P),P) \ar[r]^-f   &  (h \colon C \to \Cod^n (Q),Q)
}
\]
in $C \downarrow \Cod^n$ that is a map $f\colon P\to Q$ in  $\Ext_\Pc^n(\Cs_\Pc)$  such that $  \Cod^n  (f) = h $ and $P \in \Pres^n_\Pc(\Cs_\Pc)$. We shall denote by 
\[
\xymatrix{ (1_C \colon C \to C= \Cod^n(Q^\star),Q^\star) \ar[r]^-{\Qc^\star} & (h \colon C \to \Cod^n (Q),Q)
}
\]
any weak terminal object in $\Ps_C^n\downarrow \Qc$ (if any exists).

With the convention that $\Cod^0=1_{\Cs_\Pc}$, the result is true when $n=0$. Let $n>0$ and let us suppose that for $0\leq k < n$ the result is true. Let us write
\[
\xymatrix{
\delta_{n-1} (Q) \ar@{}[r]|-= & Q_1 \ar[r]^-q &Q_0
}
\]
One consider a weak terminal object 
\[
\Qc_0^\star =(1_C \colon C \to C=\Cod^{n-1}(Q_0^\star),Q_0^\star)
\]
in $\Ps^{n-1}_C \downarrow \Qc_0$ where
\[
\Qc_0=(h\colon C \to \Cod^{n-1}(Q_0),Q_0)
\]
and form a commutative diagram
\[
\xymatrix{
Q_1^\star  \ar@/^/[rrd]^-{\Qc_1^\star}    \ar@/_/[rdd]_-{q^\star}    \ar@{.>}[rd]|-{u}  & \\
&Q_0^\star \times_{Q_0} Q_1   \ar[d]_-{}    \ar[r]^-{}  & Q_1  \ar[d]^-q  \\
&Q_0^\star   \ar[r]_-{\Qc_0^\star}  & Q_0
}
\]
where $u\colon Q_1^\star   \to Q_0^\star \times_{Q_0} Q_1$ is a $1$-fold $\Pc^{n-1}$-projective presentation of $Q_0^\star \times_{Q_0} Q_1$.  Of course $Q^\star$, determined by 
\[
\xymatrix{
\delta_{n-1} (Q^\star) \ar@{}[r]|-= & Q_1^\star \ar[r]^-{q^\star} &Q_0^\star,
}
\]
is in $\Pres^{n}_\Pc(C)$. Finally, it is easy to check that $\Qc^\star \colon Q^\star \to Q $, given by 
\[
\delta_{n-1} (\Qc^\star )=(\Qc_1^\star,\Qc_0^\star) \colon q^\star \to q,
\]
is an object
\[
\xymatrix{
(1_C \colon C\to C=\Cod^{n} (Q^\star),Q^\star) \ar[r]^-{\Qc^\star} & (h \colon C \to \Cod^{n} (Q),Q)
}
\]
which is weakly terminal in $\Ps_C^n \downarrow \Qc$.
\end{proof}
In fact, one has the following situation 
\begin{equation}\label{globalsituation}
\vcenter{
\xymatrix{
\Ps^n_C \ar[r]^-{\tilde Q_C} \ar@{->}[d]_-i \ar@{}[rd]|-= & \Pres^n_\Pc(C) \ar@{->}[d]\\
C\downarrow \Cod^n \ar[r]^-{Q_C}   \ar[d]_-{P_C}  & \ar[d]^-{\Cod^n}  \Ext^n_\Pc(\Cs_\Pc)\ar[r]^-{\Gs^{\Gamma,\Pc}_n}   &  \Gp(\Xs)\\
   {\bf 1}  \ar[r]_-C  \ar@{}[ru]|-{\Rightarrow} & \Cs_\Pc
}
}
\end{equation}
where $\tilde Q_C$ is an isomorphism and the inclusion $i$ initial. We are going to  show that $\Gs^{\Gamma,\Pc}_n$ restricted to $\Pres_\Pc^n(C)$ has a limit. For this, we first prove the following

\begin{lemma}\label{Baer-Invariant2} Let $1\leq n$, $C \in \Cs_\Pc$  and $f,g \colon P\to Q$ be two morphisms in $\Pres_\Pc^n(\Cs)$ such that $\Cod^{n}(f)=\Cod^n(g)$. Then 
\[
 \Gs_{n}^{\Gamma,\Pc} (f)= \Gs_{n}^{\Gamma,\Pc} (g)\colon\Gs_{n}^{\Gamma,\Pc}(P)\to \Gs_{n}^{\Gamma,\Pc}(Q)
\]

\end{lemma}\label{Baer}
\begin{proof} The case $n=1$ follows from Corollary \ref{corobaer}. Now let $n>1$ and let us suppose that  the result holds  for $0\leq k <n$ and any Galois structure of type \ref{GStypePP}. It is possible to construct a diagram
\begin{equation}\label{thediagram}
\xymatrix{
P' \ar[r]^-p \ar@<0.5ex>[d]^-{f'}  \ar@<-0.5ex>[d]_-{g'} & P \ar@<0.5ex>[d]^-f \ar@<-0.5ex>[d]_-g \\
Q' \ar[r]_q  & Q 
}
\end{equation}
in $\Pres_\Pc(\Cs_\Pc)$ such that 
\[
\begin{aligned}
\Cod^{n-1} \rho_0(f') &= \Cod^{n-1} \rho_0(g') ;& \delta_{0}^+(p)&=1 ; & \delta_{0}^+(q)&=1;
\end{aligned}
\]
\[
\begin{aligned}
\delta_{0}^+ (f\circ p) &= \delta_{0}^+ (q \circ f'); & \delta_{0}^+ (g\circ p)& = \delta_{0}^+ (q \circ g').
\end{aligned}
\]
Consequently  $ \Gs_{n}^{\Gamma,\Pc} (p)$ and $ \Gs_{n}^{\Gamma,\Pc} (q)$ are isomorphisms and one can  conclude using Lemma \ref{invargal2} and Lemma \ref{invargal3}. Let us give some details about the  construction of the diagram above (note that the proof is a modified version of \cite[Theorem 2.3.10]{Tomasthesis}). For $i=0 ,\hdots,n-1$, we define $P^{(i)}=(P_{S^0})_{S\subseteq i}$, $Q^{(i)}=(Q_{S^{0}})_{S\subseteq i}$, $f^{(i)}=(f_{S^0})_{S\subseteq i}$ and $g^{(i)} =(g_{S^0})_{S\subseteq i}$ in  $\Ext^i_\Pc(\Cs_\Pc)$. One easily check that $\delta_i^+(P^{(i+1)}) = P^{(i)}$ and $\delta_i^+(Q^{(i+1)}) = Q^{(i)}$ and so on. Then we construct inductively some commutative diagrams $D_f^{(i)}$ and $D_g^{(i)}$:
\[
\xymatrix{
P'^{(i)} \ar[r]^-{p^{(i)}}  \ar[d]_-{f'^{(i)}}     & P^{(i)} \ar[d]^-{f^{(i)}}\\
Q'^{(i)} \ar[r]_-{q^{(i)}}   &    Q^{(i)}
}
\xymatrix{
P'^{(i)} \ar[r]^-{p^{(i)}}  \ar[d]_-{g'^{(i)}}     & P^{(i)} \ar[d]^-{g^{(i)}}\\
Q'^{(i)} \ar[r]_-{q^{(i)}}   &    Q^{(i)}
}
\]
for $i=0,\hdots, n-1$. First, we set $D_f^{(0)} = D_g^{(0)}$
\[
\xymatrix{
P'^{(0)} \ar[r]^-{p^{(0)}}  \ar@{.>}[d]_-{f'^{(0)}=g'^{(0)}}     & P_0 \ar[d]^-{f_0=g_0}\\
Q'^{(0)} \ar[r]_-{q^{(0)}}   & Q_0
}
\]
where $p^{(0)}$ ($q^{(0)}$) is any $1$-fold $\Ec$-projective presentations of $P_0$ ($Q_0$), and $f'^{(0)}=g'^{(0)}$ is any lifting of $f_0 p^{(0)}=g_0 p^{(0)}$ along $q^{(0)}$. Then having constructed $D_f^{(i)}$, one constructs $D_f^{(i+1)}$ using the diagram:
\[
\xymatrix@=15pt{
A \ar[rrr]^-{a} \ar@/_5ex/[ddrrr]_-{\delta_i P'^{(i+1)}} \ar@{.>}[rd]^-{c} & &&. \ar[dd] \ar[dr] \ar[rr] \ar@{}[dr]|<<{\pullback} &&  . \ar[dr] \ar[dd]^(0.7){\delta_i P^{(i+1)} }\\
&B \ar[rrr]^-{b}  \ar@/_5ex/[ddrrr]_-{\delta_i Q'^{(i+1)}} &&&.  \ar@{}[dr]|<<{\pullback} \ar[dd]  \ar[rr] &&.\ar[dd]^-{\delta_i Q^{(i+1)}} \\
&&& P'^{(i)} \ar[dr]  \ar[rr]&&P^{(i)} \ar[dr] \\
&&&& Q'^{(i)} \ar[rr]& & Q^{(i)}
}
\]
where  $a$ and $b$ are $1$-fold $\Ec^i$-projective presentations and $c$ is a suitable lifting. One constructs $D_g^{(i)}$ similarly. Finally, one obtains the diagram \ref{thediagram} via
\[
\xymatrix{
P'^{(n-1)} \ar@<0.5ex>[d]  \ar@<-0.5ex>[d] \ar[r]_-{\delta_0(P')} \ar@{.>}@/^2ex/[rr]  &  P^{(n-1)}  \ar@<0.5ex>[d] \ar@<-0.5ex>[d]& \delta^-_0 (P) \ar[l]^-{\delta_0(P)} \ar@<0.5ex>[d] \ar@<-0.5ex>[d]\\
Q'^{(n-1)} \ar[r]^-{\delta_0(Q')}  \ar@{.>}@/_2ex/[rr]&  Q^{(n-1)} & \delta^-_0(P) \ar[l]_-{\delta_0(Q) }
}
\]
where the dotted arrows (which determine $p$ and $q$) are obtained as liftings once again. Note that $\Cod^{n-1} \rho_0(f')$ and $\Cod^{n-1} \rho_0(g')$ are given by $D_f^{(0)}=D_g^{(0)}$.
\end{proof}
Looking at diagram \ref{globalsituation}, one sees that 
\[
\xymatrix{
C\downarrow \Cod^n \ar[r]^-{Q_C}&\Ext_\Pc^n(\Cs_\Pc) \ar[rr]^-{\Gs_{n}^{\Gamma,\Pc}}&& \Gp(\Xs)
}
\]
has a limit if and only if 
\begin{equation}\label{diagram}
\xymatrix{
\Pres_\Pc^n(C) \ar@{->}[r]&\Ext_\Pc^n(\Cs_\Pc) \ar[rr]^-{\Gs_{n}^{\Gamma,\Pc}}&& \Gp(\Xs)
}
\end{equation}
has a limit. Then, using Lemma \ref{Baer-Invariant2},  we find that this last limit exists and we can write
\[
\lim \Gs_{n}^{\Gamma,\Pc} \circ Q_C \cong  \Gs_{n}^{\Gamma,\Pc}(P) 
\]
for $P$ any $n$-projective $\Pc$-presentation of $C$. Indeed the image of the diagram \Ref{diagram} can be viewed as a subcategory of $\Gp(\Xs)$ which is an equivalence relation. We have proved:

\begin{theorem} For $n\geq 1$, the $n^{\text{th}}$-fundamental group functor (with respect to the Galois structure $\Gamma$ and the class $\Pc$) exists and, for any object $C$ in $\Cs_\Pc$ and $P$ any $n$-fold $\Pc$-projective presentation of $C$, one has
\[
\pi_n^{\Gamma,\Pc}(C) \cong \Gs_{n}^{\Gamma,\Pc}(P).
\]
\qed
\end{theorem}

Note that, for $n = 1$, this result shows that our definition of the first fundamental group perfectly agrees with the definition given in \cite{Janelidze:Hopf}: $\pi_1(C)$ is the Galois group of some \emph{weakly universal normal extension} of $C$ (a weak initial object in $\NExt_\Gamma(C)$). It suffices to note that, for $P$ a projective presentation of $C$,  $I_1(P)$ is a weakly universal normal extension of $C$. Let us also mention that the existence of $\pi_1$ is already proved in \cite{EGoeVdL2} (In a different context and without emphasis on the pointwise character of the Kan extension).

\subsection{The fundamental group functors as satellites}

The fundamental groups functors defined as above coincide with various other homology defined in more restricted contexts. Indeed it is easy to show (using a slight modification of the arguments given in the section 6 of \cite{EGoeVdL2}) that
\[
\pi_n^{\Gamma,\Pc} = \mathrm{Ran}_{\Cod^n}(\Ker^n \circ I_n),
\]
that is $\pi_n^{\Gamma,\Pc}$ is the pointwise right satellite of $I_n$ with respect to $\Cod^n$ and $\Ker^n$:
\[
\xymatrix{
\Ext_\Pc^n(\Cs_\Pc) \ar[r]^-{I_n} \ar[d]_-{\Cod^n}& \NExt_\Gamma^n(\Cs) \ar[d]^-{\Ker^n}\\
\Cs_\Pc \ar[r]_-{\pi_n^{\Gamma,\Pc}} \ar@{}[ru]|-{\Rightarrow}& \Xs,
}
\]
so that $\pi_n^{\Gamma,\Pc}(-) \cong H_{n+1}(-,I)$ where $H_{n+1}(-,I)$ is defined as in \cite{GVdL2}.
For the sake of completeness, we give here the explanation. Let $C$ be in $\Cs_\Pc$. It is sufficient to show that a natural transformation of the form
\[
l\colon \Delta_L \Rightarrow \Ker^n \circ I_n \circ Q_C,
\]
for some $L$ in $\Xs$, factors (uniquely) through $ \overline{\gamma} Q_C\colon \Gs_n^{\Gamma,\Pc}\circ  Q_C \Rightarrow \Ker^n \circ I_n \circ Q_C $. Let $(f,P)=(f\colon C \to \Cod^n(P),P)$ be in $C\downarrow \Cod^n$ and let us consider the following morphisms in $C\downarrow \Cod^n$ determined by :
\[
\xymatrix{
\delta_{n-1}^-(P) \ar[r]^-{\eta_{\delta_{n-1}^-(P)}^{n-1}} \ar[d]_-{\delta_{n-1}(P)} & I_{n-1}(\delta_{n-1}^-(P)) \ar[d]  & 0\ar[d] \ar[l]  \\
\delta_{n-1}^+(P) \ar[r] \ar@{.}[d] & 0\ar@{.}[d]  & 0\ar@{.}[d]\ar[l]  \\
\Cod^n(P) \ar[r] & 0& 0\ar[l] \\
  &C \ar[ul]^-f \ar[u] \ar[ur] &
 }
\]
Since $I_{n-1}\delta_{n-1} I_n(P) = I_{n-1}I_n \delta_{n-1} (P)=I_{n-1}\delta_{n-1} (P)$, the naturality of $l$ gives us a commutative diagram:
\[
\xymatrix{
L \ar@{=}[rrrrr] \ar[d]^-{l_{(f,P)}}&&&&& L\ar[d] & L \ar@{=}[l] \ar[d] \\
\Ker^n(I_n(P)) \ar[rrrrr]_-{\alpha=\Ker^{n-1}(\eta^{n-1}_{\delta_{n-1}^- I_n (P)} \circ \ker(\delta_{n-1} I_n(P)))} &&&&& \Ker^{n-1}(I_{n-1} \delta_{n-1}^- I_n(P)) &0 \ar[l]
}
\]
and  consequently $l_{(f,P)}$ factors through $(\overline{\gamma}Q_C)_{(f,P)} = \overline{\gamma}_P$, the kernel of the morphism $\alpha$.
\subsection{Hopf formulae}
\paragraph{First formulae} Let $\Gamma$ be a structure of the \Ref{GStypePP}.
\begin{theorem}\label{Hopf1} Let $1 \leq n$  and $P$ be a $n$-fold $\Pc$-projective presentation. Then
\[
\Gs_{n}^{\Gamma,\Pc}(P)\cong\dfrac{ [\Dom^n(P)]_{\Gamma,0} \cap \Ker^n(P)}{[P]_{\Gamma,n}}
\]
\end{theorem}

\begin{proof} By induction. Let $n=1$ and let  $P=p\colon P_1\to P_0$. Then, one can check that all the faces in the following cube are pullbacks
\[
\xymatrix@=15pt{
[P_1]_{\Gamma_0} \cap \Ker (p)  \ar@{.>}[dd]_-{\widehat {\overline{\eta}^1_p}}     \ar[dr] \ar[rr] & &\Ker(p)  \ar[dd] \ar[dr] \\
&[P_1]_{\Gamma,0}\ar[dd] \ar[rr] && P_1\ar[dd]^-{\overline{\eta}^1_p}\\
[I_1[p]]_{\Gamma,0}  \cap   \Ker(I_1(p)) \ar[dr] \ar[rr] &&\Ker(I_1(p)) \ar[dr] \\
&[I_1[p]]_{\Gamma,0}  \ar[rr] & &I_1[p]}
\]
Indeed, the right hand  and front faces are pullbacks, since $p=I_1(p)\circ \overline{\eta}^1_p $ and $\eta_{P_1}=\overline{\eta}^1_p\circ \eta_{I_1[p]}$. It follows that the morphism $\widehat {\overline{\eta}^1_p}$ is in $\Ec$ by \hl{E2}, thus a normal epimorphism, and $\Ker(\widehat {\overline{\eta}^1_p}) = \Ker(\overline{\eta}^1_p)$. Finally, one has
\begin{equation*}
\begin{split}
\Gs^{\Gamma,\Pc}_{1} (P) &= \Gal_\Gamma(I_1(p),0)\\
&\cong  [I_1[p]]_{\Gamma,0}  \cap   \Ker(I_1(p))\\
&\cong \dfrac{[P_1]_{\Gamma,0} \cap \Ker (p)}{\Ker(\widehat {\overline{\eta}^1_p})}\\
&= \dfrac{[\Dom(P)]_{\Gamma,0} \cap \Ker(P)}{[P]_{\Gamma,1}}
\end{split}
\end{equation*}

Let $n\geq 2$ and let us assume that  the results holds for all $1 \leq k < n$ and Galois structures of type \Ref{GStypeP}. Then one has 
\begin{equation*}
\begin{split}
\Gs_{n}^{\Gamma,\Pc}(P)&= \Dom (\Gs_{n-1}^{\Gamma_1,\Pc^1}(\rho_0 (P)))  \\
&\cong \Dom \left( \dfrac{ [\Dom^{n-1}(\rho_0(P))]_{\Gamma_1,0} \cap \Ker^{n-1}(\rho_0(P))}{[\rho_0(P)]_{\Gamma_1,n-1}}\right)\\
&\cong    \dfrac{ \Dom\left( \iota^1 [\Dom^{n-1}(\rho_0(P))]_{\Gamma,1} \cap \Ker^{n-1}(\rho_0(P))\right)}{\Dom(\iota^1 [P]_{\Gamma,n})}\\
&\cong    \dfrac{[\Dom^{n-1}(\rho_0(P))]_{\Gamma,1}   \cap \Dom ( \Ker^{n-1}(\rho_0(P)))}{[P]_{\Gamma,n}}
\end{split}
\end{equation*}

Since $\Dom^{n-1}(\rho_0(P))$ is a morphism in $\Pc$ with a projective codomain, it is a split $\Ec$-extension 
and then
\begin{equation*}
\begin{split}
[\Dom^{n-1}(\rho_0(P))]_{\Gamma,1}  \cong & [\Dom( \Dom^{n-1}(\rho_0 (P)))]_{\Gamma,0}  \cap \Ker(\Dom^{n-1}(\rho_0(P)))\\
\cong & [\Dom^{n}(P)]_{\Gamma,0}  \cap \Ker(\Dom^{n-1}(\rho_0(P)))
\end{split}
\end{equation*}
Finally, one has  
\begin{equation*}
\begin{split}
 & [\Dom^{n-1}(\rho_0(P))]_{\Gamma,1}   \cap \Dom ( \Ker^{n-1}(\rho_0(P)))\\
  &\cong[\Dom^n(P)]_{\Gamma,0}  \cap \Ker(\Dom^{n-1}(\rho_0(P))) \cap  \Dom \left( \Ker^{n-1}(\rho_0(P))\right) \\
  &\cong [\Dom^n(P)]_{\Gamma,0}\cap \Ker^n(P).
 \end{split}
\end{equation*}
\end{proof}

\paragraph{Refined formulae} If the composite Galois structure $\Gamma''$ is as in \Ref{GStypeC} and such that $\Gamma$ is of kind \Ref{GStypePP}, then the descriptions of the fundamental groups can be refined, as shown in the following theorem (the case $n=1$ in a restricted context can be found in \cite{DuckertsEveraertGran}). One first establishes some lemmas.
\begin{lemma}\label{coollemma} Let
\[
\xymatrix{
 0 \ar[r] &\iota (A)  \ar[r]  & P \ar[r]^-f & Q \ar[r]& 0
}
\]
be a short exact sequence with $A$ in $\Cs$, $f$ in $\Ec^1$, $P$ in $\Ext_\Ec(\Cs)$ and $Q$ in $\NExt_\Gamma(\Cs)$.
Then one has a short exact sequence 
\[
\xymatrix{
 0 \ar[r]&A  \ar[r]  & \Ker(P) \ar[r]^-{\Ker(f)} & \Ker(Q) \ar[r]& 0
}
\]
with $\Ker(f)$ in $\Ec$, $\Ker(P)$ in $\Cs$ and $\Ker(Q)$ in $\Xs$. Moreover, one has
\[
\overline{\iota (A)}^{\NExt_\Gamma(\Cs)}_{P} =\iota( \overline{A}^\Ys_{\Ker(P)}).
\]
\end{lemma}
\begin{proof} Let us consider the following commutative diagram
\[
\xymatrix{
&( \Ker(f))^{-1}(\overline{0}_{\Ker(Q)}^{\Ys}) \ar[r]  \ar[d] \ar@{}[rd]|<<<<{\pullback}&  \overline{0}_{\Ker(Q)}^{\Ys}   \ar[d]\\ 
\Ker^2(f) \ar[r] \ar[d]^-i  \ar@/^1pc/[ur]  & \Ker (P) \ar[r]^{\Ker(f)}  \ar[d]  \ar@{}[rd]|-{(*)} & \Ker (Q) \ar[d] \\
 A \ar[r] \ar[d]   & P_1 \ar[d]_{p}  \ar[r]^-{f_1} & Q_1 \ar[d]^q \\ 
0\ar[r]&P_0 \ar[r]_-{f_0}  &Q_0 
}
\]
Since $f$ is in $\Ec^1$, its restriction $\Ker(f)$ to the kernels of $P$ and $Q$ must be in $\Ec$ and, since $q$ is $\Gamma$-normal, $\Ker(Q)$ is in $\Xs$.  But the last row being exact, $f_0$ is an isomorphism and the square $(*)$ is a pullback. This implies that $i$ is an isomorphism. To check the second assertion is easy. It suffices to note that, since $q$ is $\Gamma$-normal, one has $\br_{\Gamma,1}(q)=\iota (\overline{0}_{\Ker(Q)}^{\Ys})$ and
\begin{equation*}
\begin{split}
\overline{\iota(A)}_{P}^{\NExt_{\Gamma''}(\Cs)}&= f^{-1}(\iota (\overline{0}_{\Ker(Q)}^{\Ys})) \\
&= \iota(f_1^{-1}(\overline{0}_{\Ker(Q)}^{\Ys}))\\
& = \iota (( \Ker(f))^{-1}(\overline{0}_{\Ker(Q)}^{\Ys}))\\
&=\iota (\overline{(( \Ker(f))^{-1}(0))}_{\Ker(P)}^{\Ys}) \\
&= \iota ( \overline{A}_{\Ker(P)}^{\Ys}) .
\end{split}
\end{equation*}

\end{proof}

\begin{corollary}\label{coollemma2} Let
\[
\xymatrix{
 0 \ar[r] &\iota^n (A)  \ar[r]  & P \ar[r]^-f & Q \ar[r]& 0
}
\]
be a short exact sequence with $A$ in $\Cs$, $f$ in $\Ec^n$, $P$ in $\Ext_\Ec^n(\Cs)$ and $Q$ in $\NExt_\Gamma^n(\Cs)$.
Then one has a short exact sequence 
\[
\xymatrix{
 0 \ar[r]&A  \ar[r]  & \Ker^n(P) \ar[r]^-{\Ker^n(f)} & \Ker^n(Q) \ar[r]& 0
}
\]
with $\Ker^n(f)$ in $\Ec$, $\Ker^n(P)$ in $\Cs$ and $\Ker^n(Q)$ in $\Xs$.
\end{corollary}

\begin{theorem}\label{Hopf2}  Let $1 \leq n$  and $P$ be a $n$-fold $\Pc$-projective presentation. Then 
\[
\Gs^{\Gamma'',\Pc}_n(P)\cong  \dfrac{\overline{([\Dom^n(P)]_{\Gamma,0})}^{\Ys}_{\Dom^n(P)}\cap \Ker^n(P) }{\overline{([P]_{\Gamma,n})}^{\Ys}_{\Ker^n (P)}}
\] 
\end{theorem}

\begin{proof}
Let $n=1$. One decomposes the formula in Theorem \ref{Hopf1}. Clearly, in our case
\[
[\Dom(P)]_{\Gamma'',0}=\overline{([\Dom(P)]_{\Gamma,0})}^{\Ys}_{\Dom(P)}.
\] 
Furthermore  the arrow $\overline{\eta}_p^1$ in the proof of Theorem \ref{Hopf1}, the cokernel of $[P]_{\Gamma'',1}$, can be factorised as $h\circ g$ where $g$ is the cokernel of $[P]_{\Gamma,1}$ and $h$ the cokernel of  $\overline{0}_{\Ker(I_1 (p))}^{\Ys}$. Let us consider the following commutative diagram
 \[
 \xymatrix{
 \hat{g}^{-1}(\overline{0}^{\Ys}_{\Ker(I_1(p))}) \ar[r] \ar[d] & \Ker(p) \ar[r]^-{\ker(p)} \ar[d]_-{\hat g}& P_1\ar[d]^-g\\
\overline{0}_{\Ker(I_1(p))}^{\Ys} \ar[r] \ar[d]& \Ker(I_1(p)) \ar[r]& I_1[p] \ar[d]^-h\\
 0\ar[rr]& & (F\circ I)_1[p] 
 }
 \]
 in which all rectangles are pullbacks. Then one sees that there is an isomorphism 
 \[
 \Ker(\overline{\eta}_p^1)= \Ker(h\circ g)\cong \hat{g}^{-1}(\overline{0}_{\Ker(I_1(p))}^{\Ys})
 \] 
  between the  domains of the normal monomorphisms 
  \[
  \ker(h\circ g)\quad  \text{and}\quad {\hat g}^{-1}({\ker}({(\theta \circ \eta)_{\Ker(I_1(p))})}).
  \] 
  Since $\hat g$ is in $\Ec$, one has the following equalities:
 \begin{equation*}
 \begin{split}
\hat{g}^{-1}(\overline{0}^{\Ys}_{\Ker(I_1(p))}) &= \overline{(\hat{g}^{-1}(0))}^{\Ys}_{\Ker(p)}\\
 &= \overline{(\Ker(\hat g))}^{\Ys}_{\Ker(p)}\\
  &= \overline{(\Ker (g))}^{\Ys}_{\Ker(p)}\\
  &= \overline{([P]_{\Gamma,1})}^{\Ys}_{\Ker(p)}. 
 \end{split}
  \end{equation*}

 Let $n\geq 2$ and let us assume that  the theorem holds for all $1\leq k\leq n-1$ and all Galois structures of type \Ref{GStypeC}.
Then, in particular,
\begin{equation}\label{berkformula}
\Gs^{\Gamma'',\Pc}_n(P) \cong \dfrac{ \Dom \left(   \overline{([\Dom^{n-1}(\rho_0(P))]_{\Gamma_1,0}])}_{\Dom^{n-1}(\rho_0(P))}^{\NExt_{\Gamma''}(\Cs)}  \right) \cap \Dom (\Ker^{n-1}(\rho_0(P))  ) }{\Dom\left( \overline{([\rho_0(P)]_{\Gamma_1,n-1})}^{\NExt_{\Gamma''}(\Cs)}_{\Ker^{n-1} (\rho_0(P))} \right)}
\end{equation}
One has $[\Dom^{n-1}(\rho_0(P))]_{\Gamma_1,0}=\iota([\Dom^{n-1}(\rho_0(P))]_{\Gamma,1})$. Applying Lemma \ref{coollemma} to the short exact sequence
 \[
 \xymatrix{
 \iota([\Dom^{n-1}(\rho_0(P))]_{\Gamma,1}) \ar[r] &  \Dom^{n-1}(\rho_0(P))  \ar[rr]^-{\eta^1_{\Dom^{n-1}(\rho_0(P))}} & &I_1(\Dom^{n-1}(\rho_0(P)))
}
 \]
gives us
\[
 \overline{([\Dom^{n-1}(\rho_0(P))]_{\Gamma_1,0}])}_{\Dom^{n-1}(\rho_0(P))}^{\NExt_{\Gamma''}(\Cs)}  = \iota (\overline{([\Dom^{n-1}(\rho_0(P))]_{\Gamma,1})}^\Ys_{\Ker(\Dom^{n-1}(\rho_0(P)))}).
\]
Since $\Dom^{n-1}(\rho_0(P))$ has a projective codomain (see the proof of Theorem \ref{Hopf1}), the numerator of \Ref{berkformula} can be rewritten as
\begin{equation*}
\begin{split}
& \Dom \left(   \overline{([\Dom^{n-1}(\rho_0(P))]_{\Gamma_1,0}])}_{\Dom^{n-1}(\rho_0(P))}^{\NExt_{\Gamma''}(\Cs)}  \right) \cap \Dom (\Ker^{n-1}(\rho_0(P))  )\\
&\cong \overline{([\Dom^{n-1}(\rho_0(P))]_{\Gamma,1})}^\Ys_{\Ker(\Dom^{n-1}(\rho_0(P)))}\cap \Dom (\Ker^{n-1}(\rho_0(P))  )\\
&\cong \overline{([\Dom^n(P)]_{\Gamma,0})}^\Ys_{\Dom^n(P)}\cap \Ker(\Dom^{n-1}(\rho_0(P))) \cap  \Dom (\Ker^{n-1}(\rho_0(P))  )\\
& \cong \overline{([\Dom^n(P)]_{\Gamma,0})}^\Ys_{\Dom^n(P)}\cap \Ker^n(P).
\end{split}
\end{equation*}
It remains to rewrite the denominator of \Ref{berkformula}. We have to figure out what is the closure of $[\rho_0(P)]_{\Gamma_1,n-1}=\iota( [P]_{\Gamma,n})$ in $\Ker^{n-1}(\rho_0(P))$ with respect to $\NExt_{\Gamma''}(\Cs)$.  Starting from the short exact sequence
\[
\xymatrix{
 \iota^{n-1}([\rho_0(P)]_{\Gamma_1,n-1} )\ar[r] & \rho_0(P)  \ar[rr]^-{(\eta^1)^{n-1}_{\rho_0(P)}} && (I_1)_{n-1}(\rho_0(P))
}
\] 
and applying Corollary \ref{coollemma2} gives us a short exact sequence
\[
\xymatrix{
  [\rho_0(P)]_{\Gamma_1,n-1}  \ar[r] & \Ker^{n-1}(\rho_0(P))  \ar[rr]^-{\Ker^{n-1}((\eta^{1})^{n-1}_{\rho_0(P)})} &&\Ker^{n-1}( (I_1)_{n-1} \rho_0(P))
}
\]
where  $\Ker^{n-1}( (I_1)_{n-1} \rho_0(P))$ is $\Gamma$-normal. Finally, we get by Lemma \ref{coollemma} that 
\begin{equation*}
\begin{split}
\Dom\left( \overline{([\rho_0(P)]_{\Gamma_1,n-1})}^{\NExt_{\Gamma''}(\Cs)}_{\Ker^{n-1} (\rho_0(P))} \right) &=\Dom(\overline{(\iota[P]_{\Gamma,n})}_{\Ker^{n-1}(\rho_0(P))}^{\NExt_{\Gamma''}(\Cs)})  \\
&=  \overline{([P]_{\Gamma,n})}_{\Ker^n(P)}^{\Ys}.
\end{split}
\end{equation*}
\end{proof}

\section{Examples}

\subsection[Groups]{Groups}


As a first example, we consider a composite adjunction of the form
\[
\xymatrix{
\Gp \ar@/^/[r]^-\ab   \ar@{}[r]|-{\perp} & \Ab \ar@/^/[r]^-F \ar@{}[r]|-{\perp} \ar@/^/[l]^-\supseteq & \Fs \ar@/^/[l]^-\supseteq
}
\]
where $\Gp$ is the category of groups, $\Ab$ the category of abelian groups and $\Fs$ the torsion-free part of a \emph{hereditary} torsion theory $(\Ts,\Fs)$ in $\Ab$ ($\Mr$-hereditary for the class $\Mr=\Mono(\Cs)$ of all monomorphisms in $\Cs$). If we choose $\Pc=\Ec=\RegEpi(\Cs)$ and the classes $\Fc$ and $\Gc$ accordingly, we get a situation in  which Theorem \ref{Hopf2} can be applied. 
Let us recall from \cite{EGVdL} that for, $P$ in $\Ext_\Ec^n(\Cs)$, 
 \[
 [P]_{\Gamma,n}=\prod\limits_{I\subseteq n} [\bigcap\limits_{i\in I}\Ker(p_i),\bigcap\limits_{i\notin I}\Ker(p_i) ].
 \]
where the commutator $[-,-]$ is the classical commutator from group theory. We now need to understand what is the closure associated with the subcategory $\Fs$. First let us recall (from \cite{CDT} for instance) that the hereditary torsion theories in $\Ab$ are completely classified. They are in bijection with radicals of the form
\[
{\bf t}_\mathscr{P}(A)=\bigvee\limits_{p\in \mathscr{P}} {\bf t}_p(A)
\]
where $\mathscr{P}$ is a set of prime numbers and 
\[
{\bf t}_p(A)=\{ a\in A\,|\, \exists n\in \Nbb: \mathrm{ord}(a)=p^n \}
\]
Here $\mathrm{ord}(a)$ denotes the order of $a$. Let us fix a set $\mathscr{P}$ and let $\Fs_\mathscr{P}$ be the associated torsion-free subcategory of $\Ab$.

One can prove that the corresponding closure of a normal subgroup $K$  of a group $A$ such that $K \geq [A,A] = [A]_{\Gamma,0}$ is
 \begin{equation}\label{closure-description}
 \begin{array}{rcl}
 \overline{K}^{\Fs_\mathscr{P}}_A &=&\{a\in A\,|\, \exists m \in \langle \mathscr{P}\rangle : a^m \in K\}.
  \end{array}
  \end{equation}
where $\langle \mathscr{P} \rangle$ is the ideal generated by $\mathscr{P}$ in the (commutative) monoid $(\Nbb_0,\cdot ,1)$. 
\begin{lemma} Let $q \colon A\to B$ be a surjective homomorphism with $B$ abelian. One has
\[
q^{-1}(\bt_\Ps(B))=\{ a\in A\,|\, \exists m\in \langle \Ps \rangle: q(a)^m=1\}.
\]
\[
\xymatrix{
q^{-1}(\bt_\Ps(B)) \ar[d]  \ar[r] \ar@{}[dr]|<<<{\pullback} &\bt_\Ps(B) \ar[d]\\
A \ar[r]_-q &B
}
\]
\end{lemma}
\begin{proof} ($\subseteq$:) Let $a$ be in $q^{-1}(\bt_\Ps(B))$, that is $a\in A$ such that 
\[
q(a)=\prod_{i=1}^k b_i 
\]
for some $b_i \in \bt_{p_i}(B)$ and $p_i\in \langle \Ps \rangle$. Then, there exist $l_i\in \Nbb$ such that $b_i^{p_i^{l_i}}=1$, and, for $m=\prod_{i=1}^{k}p_i^{l_i} (\in \langle \Ps \rangle)$,
\[
q(a)^m=\prod_{i=1}^{k}b_i^m=\prod_{i=1}^{k}(b_i^{p_i^{l_i}} )^{\prod_{j\neq i}p_j^{l_j} }=1. 
\]
($\supseteq$:) We prove by induction that for all $m$ in $\langle \Ps \rangle$,
\[
\{ a\in A \,|\, q(a)^m=1\}\subseteq q^{-1}(\bt_\Ps(B)).
\] 
Since $\bt_\Ps(B)$ is a normal subgroup of $B$, the inclusion holds for $m=1$. Now, let $m>1$ be in $\langle \Ps \rangle$ and let us decompose it as
\[
m=m'\cdot p^l
\]
with $p\in \Ps$, $l\in \Nbb$ and $m'\in \langle \Ps \setminus \{p\} \rangle(\subseteq \langle \Ps \rangle)$ such that $1 \leq m'<m$. One knows from Bezout's theorem that one can find some $c',c \in \Zbb$ such that
\[
c'\cdot m'+c\cdot p^l=1.
\]
For $a\in A$ such that $q(a)^m=1$, one has
\[
(q(a)^{m'})^{p^l}=q(a)^{m'\cdot p^l}=1
\]
and
\[
q(a^{p^l})^{m'}=(q(a)^{p^l})^{m'}=q(a)^{m'\cdot p^l}=q(a)^m=1,
\]
so that, by definition,
\[
q(a)^{m'}\in \bt_p(B)\subseteq \bt_\Ps(B)
\]
and, by induction,
\[
q(a^{p^l}) \in \bt_\Ps(B).
\]
Hence,
\[
q(a)=q(a)^{c'\cdot m'+c\cdot p^l}=\underbrace{(q(a)^{m'})^{c'}}_{\in \, \bt_\Ps(B)}\cdot \underbrace{(q(a)^{p^l})^c}_{\in \, \bt_\Ps(B)} \in \bt_\Ps(B)
\]
and $a$ is in $q^{-1}(\bt_\Ps(B))$. 
\end{proof}

For $A$ a group and $P$ any $n$-fold $\Pc$-projective presentation of $A$, the formulae
\[
\pi_n(A) \cong \dfrac{\overline{([P_n]_{\Gamma,0})}^{\Fs_\Ps}_{P_n}\cap \Ker^n(P) }{\overline{([P]_{\Gamma,n})}^{\Fs_\Ps}_{\Ker^n (P)}}
\]
in Theorem \ref{Hopf2} become 
\[
\pi_n(A)\cong \dfrac{\{ k\in \Ker^n (P) \,|\,  \exists m \in \langle \mathscr{P} \rangle: k^m \in [P_n,P_n]  \}}{\{ k\in \Ker^n( P) \,|\,  \exists m \in \langle \mathscr{P}\rangle : k^m \in \prod\limits_{I\subseteq n} [\bigcap\limits_{i\in I}\Ker(p_i),\bigcap\limits_{i\in n\setminus I}\Ker(p_i) ] \}}.
\]
For $\mathscr{P}=\emptyset$, the Brown-Ellis formulae  \cite{Brown-Ellis} for the integral homology of a group are recovered:
\[
\mathsf{H}_{n+1}(A,\Zbb) \cong \dfrac{[P_n,P_n]\cap \Ker^n (P)}{\prod\limits_{I\subseteq n} [\bigcap\limits_{i\in I}\Ker(p_i),\bigcap\limits_{i\in n\setminus I}\Ker(p_i) ]} \cong \pi_n(A).
\]

\subsection[Topological groups]{Topological groups}


We consider here the adjunctions

\[
\xymatrix@=30pt{
\Gp(\Top)  \ar@/^/[r]^-\ab   \ar@{}[r]|-{\perp} & \Ab(\Top) \ar@/^/[r]^-F \ar@{}[r]|-{\perp} \ar@/^/[l]^-\supseteq & \Ab(\Haus) \ar@/^/[l]^-\supseteq}
\]
where $\Ab(\Haus)$ is the category of Hausdorff abelian topological groups. The functor $\ab$ sends a topological group $G$ on $G/[G,G]$ with the quotient topology and $F$ sends an abelian topological group $G$ on $G/\overline{\{0\}}_G$ where $\overline{\{0\}}_G$ is the topological closure of the trivial subgroup. If we fix $\Ec=\RegEpi(\Gp(\Top))$ and $\Pc$ the class of morphisms in $\Gp(\Top)$ which are split in the category $\Top$, we are in a position to apply our Theorem \ref{Hopf2} . For $(K,k)$ a normal subobject of a topological group $A$ such that $K \geq [A,A]$, one finds that
\[
\overline{K}^{\Ab(\Haus)}_A=\overline{K}^\Top,
\]
the topological closure of $K$ in $A$. Obviously, the inequality 
\[
\overline{K}^\Top \leq {q_K}^{-1}(\overline{0}^{\Ab(\Haus)}_{A/K})=\overline{K}^{\Ab(\Haus)}_A
\]
holds since $\coker(k)=q_K\colon A \to A/K$ is continuous. The converse inequality 
\[
q_K^{-1}(\overline{0}^{\Ab(\Haus)}_{A/K})  \leq \overline{K}^\Top
\]
is also valid since the map $q_K$ is open \cite[Proposition 21]{Borceux-Clementino}. Then, for $P$ a $n$-fold $\Pc$-projective presentation of a topological group $A$, we have
\[
\pi_n(A)\cong \dfrac{\overline{[P_n,P_n]}^{\Top} \cap \Ker^n(P)}{\overline{\prod\limits_{I\subseteq n}[\bigcap\limits_{i\in I} \Ker(p_i),\bigcap\limits_{i\notin I} \Ker(p_i)] }^{\Top}}.
\]

  
\subsection{Torsion theories}


Let us consider an adjunction of the form
\[
\ad{\Cs}{\Fs}{F}{\supseteq}
\]
 where $\Cs$ is an descent-exact homological category and $\Fs$ the torsion free part of a torsion theory in $\Cs$. Let us assume that the torsion theory is $\Mr$-hereditary for a class $\Mr$ that contains the protosplit monomorphisms and, moreover, that $\Cs$ has enough projective objects with respect to a class $\Pc \subseteq \Ec =\RegEpi(\Cs)$. Then, for $P$ a  $n$-fold $\Pc$-projective presentation of an object $A$ of $\Cs$, one obtains: 
 \[
\pi_n(B)\cong  \dfrac{\overline{0}^{\Fc}_{P_n}\cap \Ker^n(P)}{\overline{0}_{\Ker^n(P)}^{\Fc}}   = \dfrac{\br (P_n)\cap \Ker^n(P)}{\br (\Ker^n(P))} 
 \]
This expression becomes trivial when the torsion subcategory is closed under normal subobjects. In fact, if $\Mr$ is a pullback stable class of monomorphisms containing the normal monomorphisms, then one finds that $f^{-1}(\br Y)=\br X$ for every morphism $f\colon X\to Y$ in $\Mr$ (this is a modified version of  \cite[Proposition 3.3 (1)]{CDT}) .
This shows that many torsion theories do not give interesting invariants. However, not all torsion-free subcategories are closed under normal monomorphism. For instance, let us consider an adjunction of the type
 \[
 \xymatrix@=30pt{
\Tbb(\HComp)\ar@/^/[r]^-F  \ar@{}[r]|-{\perp} &  \Tbb(\Prof) \ar@/^/[l]^-\supseteq} 
 \]
where $\Tbb$ is a semi-abelian theory and $\Tbb(\Prof)$ the category of models of this theory in the category of profinite spaces. $\Tbb(\Prof)$ is in fact the torsion-free part of the torsion theory $(\Tbb(\ConnComp),\Tbb(\Prof))$ where $\Tbb(\ConnComp)$ is the category of connected compact Hausdorff $\Tbb$-algebras \cite{BG:Torsion}. The torsion part of the theory is closed under protosplit monomorphisms \cite{Everaert-Gran-TT} so that the reflector $F$ (which sends an algebra $A$ on $A/\Gamma_0(A)$ \cite{Borceux-Clementino}) is protoadditive. Considering the normal subobject $\langle e^{i\frac{\pi}{2}} \rangle$ of $\mathsf{S}^1$ in $\Gp(\HComp)$, one sees that $\Tc=\Gp(\ConnComp)$ is not closed under normal subobjects. So this gives an example of adjunction for  which one obtains (a priori) interesting invariants. The descriptions from Theorem \ref{Hopf2} become in that case
\[
\pi_n(B)\cong \dfrac{\Gamma_0(P_n) \cap \Ker^n(P)}{\Gamma_0(\Ker^n(P))}.
\]
Similar results can be obtained for any regular epi-reflection of an almost abelian category into one of its full replete subcategory.

\section*{Acknowledgments}

The author wishes to thank  Maria Manuel Clementino and Marino Gran for their useful comments on the text.

\bibliography{tim}
\bibliographystyle{amsplain}

\end{document}